\documentclass[12pt]{amsart}
\usepackage{graphicx}
\usepackage{amssymb}
\usepackage{amsmath}
\usepackage{amsfonts}
\usepackage{amsthm}
\usepackage[english]{babel}
\usepackage[latin1,ansinew]{inputenc}
\usepackage[square,comma,numbers]{natbib}
\usepackage{a4wide}
\usepackage{color}
\usepackage{academicons}
\usepackage{hyperref}

\newcommand{\be}{\begin{eqnarray}}
	\newcommand{\ee}{\end{eqnarray}}

\swapnumbers
\newtheorem{theo}{Theorem}[section]
\newtheorem{remark}[theo]{Remark}
\newtheorem{lemma}[theo]{Lemma}
\newtheorem{coro}[theo]{Corollary}
\newtheorem{defi}[theo]{Definition}
\newtheorem{prop}[theo]{Proposition}
\newcommand{\R}{\mathbb R}
\newcommand{\C}{\mathbb C}
\newcommand{\N}{\mathbb N}
\newcommand{\mG}{{\mathcal G}}
\newcommand{\mD}{{\mathcal D}}
\newcommand{\mB}{{\mathcal B}}
\newcommand{\mL}{{\mathcal L}}

\newcommand{\mJ}{{\mathcal J}}

\newcommand{\mT}{{\mathcal T}}
\newcommand{\mU}{{\mathcal U}}
\newcommand{\mW}{{\mathcal W}}
\newcommand{\mE}{{\mathcal E}}
\newcommand{\mS}{{\mathcal S}}
\newcommand{\mF}{{\mathcal F}}

\newcommand{\lxi}{{\langle \xi \rangle}}

\newcommand{\eps}{\epsilon}

\newcommand{\Op}{\operatorname{Op}}
\newcommand{\Rep}{\operatorname{Re}}

\makeatother

\begin{document}
	
	\title[Shock profiles in dissipative hyperbolic-hyperbolic systems]
	{Non-linear stability of shock profiles in dissipative hyperbolic-hyperbolic systems}
	
\author[Sroczinski]{Matthias Sroczinski$^*$}
\address{$^*$University of Konstanz, Konstanz, DE}
\email{matthias.sroczinski@uni-konstanz.de, ORCID: \href{https://orcid.org/0000-0002-5472-2741}{0000-0002-5472-2741}}
\thanks{$^*$Corresponding author, matthias.sroczinski@uni-konstanz.de. Research of M.S. was supported
	under DFG grant no. FR 822/11-1
	(SPP-2410). }

\author[Zumbrun]{Kevin Zumbrun$^\dagger$}
\address{$^\dagger$Indiana University, Bloomington, IN 47405}
\email{kzumbrun@iu.edu}
\thanks{$^\dagger$Research of K.Z. was partially supported
under NSF grants no. DMS-2154387 and DMS-2206105, the DFG German research network,
and the Humboldt Research Prize.}

\begin{abstract}
We give the first proof of nonlinear stability for smooth shock profiles
of second-order dissipative hyperbolic-hyperbolic systems under the assumption of spectral stability,
showing stability of smooth small-amplitude profiles in dimensions greater than or equal to two.
This class of systems notably includes the two types of causal viscous relativistic gas models
introduced respectively by Freist\"uhler-Temple and Bemfica-Disconzi-Noronha, and (the equivalent
second-order form of) a class of first-order numerical relaxation systems generalizing the well-known 
Jin-Xin relaxation schemes.
A significant technical innovation is a new para-differential type of nonlinear damping estimate
similar to that used by the first author to study stability of constant states,
allowing the treatment of systems far from the symmetric structure required for the standard 
``Kawashima-type'' energy estimates that are typically used for that purpose.
\end{abstract}

	
    \maketitle
    
     \tableofcontents
    
    \section{Introduction}

 In this work, we consider time-asymptotic stability of smooth shock profiles of 
 hyperbolic-hyperbolic systems
\begin{equation}\label{hyp-hyp}
    \sum_{j=0}(f^j(u))_{x_j}=\sum_{j,k=0}^d (B^{jk}(u)u_{x_j})_{x_k},
\end{equation}
arising in the relativistic dynamics of barotropic fluids incorporating viscosity and or heat conduction 
\cite{FT14,FT18,BDN18,BDN19}, with $x_0$ denoting time and $(x_1,\dots, x_d)$ space variables.
Written in more standard form distinguishing time and space variables, 
these appear as a {\it hyperbolic regularization}
\begin{equation}\label{hypreg}
\begin{aligned}
	\partial_t g(u) + \sum_{i=1}^d \partial_{x_i}f^i(u)&= 
	\sum_{i,j=1}^d \partial_{x_j}(B^{ij}\partial_{x_i}u)
	-\partial_t (\mathcal Au_{t
	})+ \sum_{i=1}^d \partial_{x_i}(C^i_0 \partial_t u) \\
	&\quad+ \sum_{i=1}^d \partial_t(C^i_1 \partial_{x_i} u),\\
	\qquad u|_{t=0}=\phi,& \quad \partial_t u|_{t=0}=\psi
\end{aligned}
\end{equation}
of a first-order system of conservation laws
\begin{equation}\label{inv}
	\partial_t g(u) + \sum_{i=1}^d \partial_{x_i}f^i(u)=0,
\end{equation}
analogous to parabolic regularizations in the classical theory of viscous shock waves.

That is, as described in \cite{FT14}, relativistic transport effects are modeled here not by 
classical diffusion but by a second-order wave operator featuring finite propagation speed.
Indeed, subluminality, or boundedness of first-order characteristic speeds by the speed of 
light, in this context corresponding to a subcharacteristic condition in
the sense of Whitham \cite{W74}, is seen to be an essential ingredient in
stability of constant solutions.
This, together with other {\it dissipativity conditions} identified by Freist\"uhler and Srocinski
has been shown to yield not only spectral but also linear and nonlinear stability of
constant states \cite{S19,S24,FS21,FS25}. 

Starting with \cite{FT14}, the existence of smooth shock profiles for hyperbolically regularized relativistic shock waves
has been intensively studied; cf. \cite{F25} and the references therein.
A very natural, and physically important problem is then to investigate the stability of those profiles,
as we do here for all dissipative hyperbolic-hyperbolic systems, not necessarily arising from
either relativistic equations or hyperbolic regularization.
More precisely, in the spirit of \cite{ZH98,Z01,Z04,Z07} in the classical case of 
hyperbolic-parabolic shocks, we show that {\it spectral stability}, defined in terms of an 
{Evans function condition}, {\it implies linear and nonlinear stability of small-amplitude
dissipative hyperbolic-hyperbolic shock profiles in dimensions $d\geq 2$.}

\subsection{Main results}\label{s:results}
Specifically, let  $\bar u(x_1)$ be a planar shock solution of \eqref{hypreg} connecting
endstates $u_\pm$ of a Lax-type shock solution of \eqref{inv}:
\begin{equation} 
	\label{asym}
\hbox{\rm
$|\partial_x^k \bar u(x) - u_\pm| \leq e^{-\delta |x|}$ for $x\gtrless 0$, $0\leq k\leq K$,}
\end{equation}
$K$ sufficiently large.
Then, for dimensions $d\geq 2$, under a standard technical assumption \cite{Z01,Z04,Z07,GMWZ05}
on the spectral structure of the Fourier-Laplace symbols of the linearized equations about 
endstates $u_\pm$,
we show for $\| (\phi-\bar u), \psi\|_{W^{1,1}\cap H^s}$ sufficiently small, $s$ sufficienty large,
there exists a global solution $u(t,x)$, $t\geq 0$ of \eqref{hypreg}, satisfying for all $t\geq 0$
and $p\geq 2$
\begin{equation}\label{decayintro}
\begin{aligned}
	\|u(t,\cdot) - \bar u\|_{H^s} &\leq C(1+t)^{-\frac{d-1}{4}},\\
	\|u(t,\cdot) - \bar u\|_{L^p} &\leq C(1+t)^{-\frac{d-1}{2}(1-1/p)};
\end{aligned}
\end{equation}
see Theorem \ref{the:main}, Section \ref{s:5}.
That is, for smooth $L^1$-localized data, we obtain decay at the rate of a $(d-1)$-dimensional 
heat kernel

The standard technical condition mentioned above, having to do with structure of ``glancing sets''
\cite{M00,M01}, is satisfied always for rotationally invariant systems or in dimension $d=2$.
It can be removed by the elegant argument of Nguyen \cite{N09} at the expense of $(1+t)^{1/4}$
decay in $L^p$ and $H^s$ estimates \cite[Thm. 1.8]{N09}.
Similarly, undercompressive and overcompressive shocks may be treated as in \cite{Z04}, under
appropriate assumptions on the shock profile (the latter with some loss of time-algebraic decay)
\cite[Result 5, p. 337]{Z04}.

\subsection{Discussion and open problems}\label{s:disc}
The above results are the first nonlinear stability results for shock waves of hyperbolic-hyperbolic
systems, reducing the study of stability to verification of spectral stability conditions.
They apply to smooth profiles of the relativistic gas dynamic models introduced both by
Freist\"uhler-Temple \cite{FT14,FT18} and Bemfica-Disconzi-Noronha \cite{BDN18,BDN19}.

Nonlinear stability of constant states for these models has been shown in \cite{FS21,FS25,S24},
for which spectral stability, determined by analysis of the Fourier symbol of the associated
linearized operator, was proven in the course of the analysis.
A very interesting open problem is to verify multi-D spectral stability of small-amplitude 
hyperbolic-hyperbolic shock profiles lying sufficiently near a stable constant state, as done for 
the hyperbolic-parabolic case in \cite{FSz02,PZ04,B24} (1-D) and \cite{FSz10} (multi-D), 
and in the 1-D relaxation case in \cite{H03,PZ04}.
A corresponding 1-D analysis for hyperbolic-hyperbolic shocks has recently been carried 
out by B\"arlin \cite{B24}.

\subsubsection{Method of proof and relations to previous work}\label{s:method}
As noted in \cite{W74,JX95} and elsewhere (e.g., \cite{FS21,B24}), the class of hyperbolic systems
\eqref{hyp-hyp} is closely related to the topic of relaxation.
Indeed, \eqref{hyp-hyp} may be considered as a second-order relaxation system with equilibrium manifold,
expressed in the phase space $(u,\partial_t u)$
$$
\partial_t g(u) + \sum_i \partial_{x_i} f^i(u)=0
$$
corresponding to the first-order hyperbolic system \eqref{inv}.
For the semilinear case, with $g(u)$ linear, \eqref{hyp-hyp} is in fact equivalent to
a first-order relaxation system of form generalizing that introduced in \cite{JX95} for
the purpose of numerical approximation. See Section \ref{s:regsec} for further discussion.

The general setting of our analysis thus neighbors those of shock stability
for second-order parabolically regularized, or ``viscous'' conservation laws, 
and first-order relaxation systems.
Accordingly, our method of analysis follows the general framework set up in
\cite{ZH98,Z01,Z04,Z07} for parabolic conservation laws, and extended in 
\cite{MaZ02,MaZ05,KwZ09,Kw11} to relaxation systems,
consisting of bounded-frequency resolvent estimates obtained by conjugation to constant coefficients,
together with high-frequency estimates obtained by WKB-type methods, converted to linear
time-evolution estimates by inverse Laplace transform/semigroup bounds,
and augmented with ``nonlinear damping''-type energy estimates sufficient to close a nonlinear iteration.

Multi-D stability of small-amplitude shock profiles for symmetrizable first-order 
relaxation  systems was shown by Kwon \cite{Kw11} following the pointwise resolvent
bound approach of \cite{Z01,Z04}.
Here, we follow an alternative approach by Kreiss symmetrizer estimates carried out in \cite{GMWZ05}
in the analysis of multi-D shock stability for parabolic conservation laws with Laplacian viscosity,
that is particularly convenient for our needs.
Indeed, we are able to directly ``plug in'' to that framework a number of our estimates 
in order to obtain the ultimate bounds we require.

Where the arguments largely differ is in the high-frequency regime and in the nonlinear damping estimates
used to control regularity.
For the semilinear parabolic case treated in \cite{GMWZ05}, of course, there is no issue with regularity.
However, for the more general quasilinear and or partial parabolic cases treated in \cite{Z04},
and in the relaxation analysis of \cite{Kw11}, these two types of estimate are treated in similar
fashion using assumptions of symmetrizability and Kawashima-type coupling on the structure
of the underlying nonlinear system, {\it assumptions that manifestly do not hold}
in the relativistic gas applications of our main interest.
Here, following a strategy proposed in \cite{Z24},
the estimates are obtained {\it without structural assumptions} on the equations, other
than dissipativity assumptions on the Fourier symbol of the constant state,
using paradifferential methods like those introduced in \cite{S24}, 
for the treatment of stability of constant solutions.

The latter is perhaps the most significant technical innovation of the paper, opening the door
to a variety of new applications in all three areas: parabolic conservation laws, relaxation,
and hyperbolic-hyperbolic systems, among others.
In particular, through the connection to semilinear relaxation systems described
in Appendix \eqref{s:connection}, our results imply also a number of new results on multi-D shock
stability in the semilinear relaxation setting, removing restrictive requirements of simultaneous 
symmetrizability.  

It is worth emphasizing that stability of constant solutions at endstates $u_\pm$ is {\it necessary}
for spectral stability of shock profiles, by the standard fact that essential spectrum boundaries for
asymptotically constant traveling waves are given by the rightmost envelope of the dispersion relations
at the endstates.
For small-amplitude shocks this implies by continuity stability of constant states all along the profile,
which is the condition needed to carry out the type of paradifferential damping estimate done here.
Thus, {\it at least for small shocks, this construction is essentially optimal},\footnote{
Assuming, as here, and for many types of systems, that stability of constant solutions persists
under perturbation, in some cases involving additional structural assumptions such as 
strict hyperbolicity.} 
sufficing to treat nonlinear and high-frequency damping estimates for general systems,
whenever spectral stability holds.

Moreover, as described in Appendix \ref{s:regsec}, even in the case that simultaneous
symmetrizability holds, we gain a bit more information, obtaining multi-D relaxation stability 
assuming $L^1$-localization only of the macroscopic, or ``equilibrium'' variable $u$ and 
{\it spatial derivatives} of the ``microscopic'', or  transient variable $v$, rather than on $u$ and $v$ 
together, a scenario reminiscent of behaviour in modulation of periodic waves \cite{JNRZ14},
with here $v$ playing the role of phase shift and $u$ of localized perturbation.

\subsubsection{Future directions}\label{s:future}
Both for our analysis here and the relaxation analysis of \cite{Kw11}, the restriction to small-amplitude
shocks is needed only for high-frequency resolvent and nonlinear damping estimates, the latter
implying the former as described in \cite{Z24}. A very important open problem for both first- and
second-order relaxation is the development of nonlinear damping estimates for large-amplitude shocks in multi-D, which would
then immediately give nonlinear stability for arbitrary amplitudes, assuming that spectral
stability holds. Further discussion of this problem may be found in \cite{Z24}.

Another very interesting open problem is the companion problem of
verification of spectral stability, either 1- or multi-D,
for {\it large-amplitude} profiles of hyperbolic-hyperbolic systems, either numerically or 
analytically. For the specific systems of relativistic isentropic gas dynamics, one might attempt to apply
the energy methods of \cite{MN85} (barotropic pressure law) and \cite{MW10,HH20} (general pressure law)
used to show analytically large-amplitude stability in the 1-D nonrelativistic case.
For large multi-D profiles, one could look to the numerical Evans function studies of
\cite{HLZ17,YZ25} and references therein.

We note that our multi-D stability argument does not automatically imply 1-D stability, since the linearized decay rates are too slow (in fact non-decaying) to close the standard nonlinear iteration. 
These rates being dominated by our estimates of the evolution of the shock front,
a very interesting problem is to remove this slower-decaying part by phase decomposition
as in \cite{MaZ02,MaZ05}, establishing 1-D nonlinear stability for smooth hyperbolic-hyperbolic
shock fronts, again under the assumption of spectral stability.
A problem of interest in any of the three settings, parabolic conservation laws and first- and second-order
relaxation, is to separate out the phase (``front behavior'') in the multi-D setting, showing faster
decay in the far field.
It is expected (see discussions of \cite{Z01,Z04}) that the front also decays faster 
than a $(d-1)$-dimensional heat kernel when properly analyzed, which would give an improved
decay rate on the entire solution (again, somewhat reminiscent of the periodic modulation \cite{JNRZ14}).

Other interesting directions for future exploration are:

$\bullet$ Analysis of the small relaxation-time limit, analogous to the small viscosity limit
treated in \cite{GMWZ04}.

$\bullet$ In the relativistic setting, analysis of the classical limit as light speed goes to infinity.

$\bullet$ Singular points, subshocks, and time-asymptotic behavior,
singular points having been shown in \cite{Pel24} to arise in the traveling-wave ODE
for hyperbolic-hyperbolic shocks in the models of \cite{BDN18,BDN19},
similarly as in shallow-water flow \cite{YZ25} and extended thermodynamics \cite{We95}.

\subsection{Plan of the paper}
In Section \ref{s:2}, we describe the setting of the problem, and introduce
the main tools to be used in the analysis.
In Section \ref{s:3}, we derive linearized time-evolutionary bounds,
assuming the necessary resolvent bounds, using the basic framework of \cite{GMWZ05}.
In Section \ref{s:4}, we derive the corresponding  resolvent bounds. For small and intermediate (i.e., bounded)
frequencies we use Kreiss symmetrizers and conjugation to constant coefficients. In Section \ref{s:5}, we first establish the crucial nonlinear damping estimate, using para-differential techniques like those used in \cite{S24} for the treatment
of stability of constant states, as suggested speculatively for small shocks in \cite{Z24},
 and then combine the above ingredients to 
complete the proof of nonlinear stability by an iteration scheme following the general approach
of \cite{Z04,Z07}.
In Appendix \ref{s:paradiff}, we recall for completeness the paradifferential calculus tools
needed for the analysis.
In Appendix \ref{s:connection}, we discuss further the connection between hyperbolic-hyperbolic
systems, first-order relaxation systems, and the Jin-Xin relaxation scheme \cite{JX95}.

\section{Setting}\label{s:2}

\subsection{Notation.} We write $\mF$ for the Fourier-transform on $L^2(\R^d)$ and set 
$$\Lambda=\mF\lxi\mF^{-1},~\lxi=(1+|\xi|^2)^{\frac12}.$$
We denote by  
$$
	\|v\|_m=\|v\|_{H^m}=\|\Lambda^m v\|_{L^2},~m \in \R,~~\|v\|:=\|v\|_0
$$
the standard norms on $L^2$-based Sobolev spaces $H^m(\R^d)$, also including non-integer indices, as well as
$$\langle v_1,v_2 \rangle_m=\langle v_1,v_2\rangle_{H^m}$$
for the corresponding scalar products. For $V=(v_1,v_2) \in H^{m_1} \times H^{m_2}$, $m_1,m_2 \in \R$, we write
$$\|V\|_{m_1,m_2}^2=\|v_1\|_{m_1}^2+\|v_2\|_{m_2}^2.$$
For a Matrix $K \in \C^{n \times n}$, we write $K^*$ for its adjoint and $\Rep K=\frac12(K+K^*)$ for its real (hermitian) part. We denote by $\operatorname{Gl}_n$ the space of invertible $\C^{n \times n}$-matrices.\\
For a semi-simple eigenvector $\mu$ of $K$ (of multiplicity $\alpha$) and an (orthogonal) basis
$r_1,\ldots,r_\alpha$ of the corresponding eigenspace $E_\mu$, we call the-$\C^{n \times \alpha}$ matrix
$$R_\mu=\begin{pmatrix} r_1 & \ldots & r_\alpha\end{pmatrix}$$
a (orthogonal) right projector on $E_\mu$ of $\mu$, and 
$$L_\mu=\begin{pmatrix} l_1 \\ \vdots \\l_\alpha \end{pmatrix}$$
a dual left projector, where $l_1,\ldots, l_\alpha \in \C^{1 \times n}$ is the basis of $E_\mu'$ dual to $r_1,\ldots,r_\alpha$. 

\subsection{Assumptions on the equations.} Throughout this paper, we consider the Cauchy problem \eqref{hypreg} with smooth coefficient matrices $\mathcal A, B^{jk}, C^j_0, C^j_1, A^0, A^j$, $j,k=1,\ldots,d$.

With $A^j(u)=Df^j(u)$, $C^j=C^j_0+\tilde C^j_1$, we define the Fourier-symbols with respect to the space variables as
\begin{align*}B(u,\xi)=\sum_{j,k=1}^dB^{jk}\xi_{j}\xi_k,~~C(u,\xi)=\sum_{j=1}^dC^{j}(u)\xi_j,\\
	A(u,\xi)=\sum_{j=1}^dA^j(u)\xi_j,~~ (u,\xi) \in \R^n \times \R^d.
\end{align*}
We also use $A^0(u)=Dg(u)$ and $\tilde M(u)=\mathcal A(u)^{-1}M(u)$ for any matrix  $M \in \C^{n \times n}$.
Our notion of hyperbolicity is as follows.

{\bf Condition (H).} Both sides of \eqref{hypreg} are hyperbolic in the sense that:

\begin{itemize}
	\item[(H$_B$)] For all $u \in \R^n$, $\mathcal A(u)$ is symmetric positive definite and  the matrix family
	$$i\begin{pmatrix}
		0 & I\\
		-\tilde B(u,\omega) & i\tilde C(u,\omega)
	\end{pmatrix},~~u \in \R^n, \omega \in S^{d-1},$$
	permits a symbolic symmetrizer.
	\item[(H$_A$)] For all $u \in \R^n$, $A^0(u)$ is symmetric positive definite and the matrix family
	$$A^0(u)^{-1}A(u,\omega),~~u \in \R^n, \omega \in S^{d-1},$$
	permits a symbolic symmetrizer.
\end{itemize}

Recall that for a smooth matrix family $K: \R^n \times S^{d-1} \to \C^{n \times n}$ a symbolic symmetrizer for $K$ is a smooth family of hermitian uniformly positive definite matrices $S \in C^\infty(\R^n \times S^{d-1}, \C^{n \times n})$ bounded as well as all of its derivatives such that $(S(u,\omega)K(u,\omega))^*=S(u,\omega)K(u,\omega)$.

Note that (H$_B$) implies the local-well-posedness of the Cauchy-problem associated to \eqref{hyp-hyp} in $L^2$-based Sobolev-spaces.
\begin{prop}[\cite{T91}, Proposition 5.3.A\footnote{Taylor assumes strict hyperbolicity of the smybol $\mB(u,\omega)$. However, inspecting the proof reveals that this property is only used to construct a symbolic symmetrizer for $\mB(u,\omega)$. Thus, the weaker assumption (H$_B$) is sufficient.}]
	\label{prop:wp}
	\label{prop:locwell}
	Let $s>d/2+1$. For all $(\phi,\psi) \in H^{s+1} \times H^s$, there exists $T>0$ such that \eqref{hypreg} admits a unique solution $u \in C^l([0,T),H^{s+1-l})$, $l=0,\ldots,s$. Additionally,  the mapping $(\phi,\psi) \to (u(t),u_t(t))$ is continuous on $H^s\times H^{s+1}$ and the solution persists as long as
	$\|(u,u_t)\|_{s+1,s}$ is bounded. 
\end{prop}

\begin{remark}
	(H$_A$), (H$_B$) in particular implies that all eigenvalues of $A^0(u)^{-1}A(u,\omega)$ ($i\mB(u,\omega)$) are real  and semi-simple. On the other hand, the latter implies the existence of a symbolic symmetrizer for $(A^0(u))^{-1}A(u,\omega)$ ($i\mB(u,\omega)$) if additionally the multiplicities of the eigenvalues do not depend on $(\omega,u)$. In this case, we call $A^0(u)^{-1}A(u,\omega)$ ($i\mB(u,\omega)$) hyperbolic with constant multiplicities.
\end{remark}

Lastly, we consider the criteria for uniform stability of a homogeneous reference state as introduced in \cite{FS21,S24}.

{\bf Condition (D).} At a constant state $u_*$, the matrices $A^j(u_*), B^{jk}(u_*)$ satisfy three conditions:
\begin{itemize}
	\item[\quad(D1)] 
	There exists a symbolic symmetrizer $S$ for $(A^0)^{-\frac12}A(A^0)^{-\frac12}$ such that for  every $\omega\in S^{d-1}$, all restrictions, as a quadratic form, of
	\begin{align*}
		W_1&=S(u_*,\omega)^{1/2}A^0(u_*)^{-\frac12}\Big(- B(u_*,\omega)
		+\mathcal A(A^0(u_*)^{-1}A(\bar u_*,\omega))^2\\
		&\quad \quad\quad-C(u_*,\omega)A^0(u_*)^{-1} A(u_*, \omega)\Big)A^0(u_*)^{-\frac12}S(u_*,\omega)^{-1/2}
	\end{align*}
	on the eigenspaces $E=J_E^{-1}(\C^n)$ of
	\begin{equation} 
		\label{w00}
		W_0(\omega)=-S(u_*, \omega)^{1/2} A^0(u_*)^{-\frac12}A(u_*,\omega)A^0(u_*)^{-\frac12} S(u_*,\omega)^{-1/2}
	\end{equation}
	are uniformly negative in the sense that
	$$
	J_E^*\Rep(W_1)J_E\le -\bar c\
	J_E^*J_E
	\quad\text{with one }\bar c>0.
	$$
	\item[(D2)]
	There exists a symbolic symmetrizer $\mathcal S$ for $i\mathcal B$ such that   
	for every $\omega\in S^{d-1}$, all restrictions, as a quadratic form, of
	$$
	\mathcal W_1=\mathcal S(u_*,\omega)^{1/2}
	\begin{pmatrix}
		0&0\\
		-i\tilde A(u_*,\omega)&- \tilde A^0(u_*,\omega)
	\end{pmatrix}
	\mathcal S(u_*,\omega)^{-1/2}
	$$
	on the eigenspaces $\mathcal E=\mathcal J_{\mathcal E}^{-1}(\C^{2n})$ of
	$$
	\mathcal W_0=
	\mathcal S(u_*,\omega)^{1/2}
	\mathcal B(u_*,\omega)
	\mathcal S(u_*,\omega)^{-1/2}
	$$
	are uniformly negative in the sense that
	$$
	\mathcal J_{\mathcal E}^*(\Rep \mW_1)\mathcal J_{\mathcal E}\le -\bar c\ \mathcal J_\mathcal E^*\mathcal J_\mathcal E
	\quad\text{with one }\bar c>0.
	$$
	\item[(D3)] All solutions
	$
	(\lambda,\xi)\in\C\times \R^d\setminus\{0\}$
	of the dispersion relation,
	\begin{equation}
		\label{disp}
		\det(\lambda^2 \mathcal A(u_*)+B(u_*,\xi)-i\lambda C(u_*,\xi)+\lambda A^0(u_*)+iA(u_*,\xi))=0,
	\end{equation} for \eqref{hypreg} at $u_*$ have
	{\rm Re}$(\lambda)<0$.
\end{itemize}

\begin{defi}
	\label{def:stable}
	We call $u_* \in \mU$ stable if $u_*$ satisfies (D). 
\end{defi}

Note that stability of $u_*$ in the sense above indeed implies (nonlinear) asymptotic stability under small perturbations of the rest-state $u_*$ (\cite{S24}).

\begin{remark}
	\label{rem:d}
	(D1), (D2), (D3) arise from the fact that the following holds for solutions $(\lambda,\xi) \in \C \times \R^{d}\setminus\{0\}$ of \eqref{disp} in polar coordinates $\xi=\rho\omega$, $\rho>0$, $\omega \in S^{d-1}$:
	\begin{itemize}
		\item[(i)] For $\rho$ sufficiently small $\lambda=\lambda(\xi)$ either satisfies
		$$\lambda(\xi)=\lambda_{0}+O(\rho)$$
		for some eigenvalue $\lambda_0$ of $-\tilde A^0$ or
		$$\lambda(\xi)=i\lambda_1(\omega)\rho+\lambda_{2}(\omega)\rho^2+O(\rho^3),$$
		where $\lambda_1(\omega)$ is an eigenvalue of $W_0(\omega)$ and $\lambda_2(\omega)$ is an eigenvalue of $(J_E^*W_1J_E)(\omega)$.
		\item[(ii)]
		For $\rho$ sufficiently large
		$$\lambda(\xi)=i\mu_1(\omega)\rho+\mu_2(\omega)+O(\rho^{-1}),$$
		where $\mu_1(\omega)$  is an eigenvalue of $\mathcal W_0(\omega)$ and $\mu_2(\omega)$ is an eigenvalue of $\mathcal (J_{\mE}^*\mW_1\mJ_{\mE})(\omega)$.
	\end{itemize}
	In particular
	\begin{equation}
		\label{imp}
		|\Im(\lambda(\xi))|\le C|\xi|.
	\end{equation}
	for some $C>0$ independent of $\xi$ and if (D) holds, we get
	\begin{equation}
		\label{rep}
		\Re \lambda(\xi) \le -c\kappa(|\xi|)~\text{ for }~\kappa(\rho)=\frac{\rho^2}{1+\rho^2}.
	\end{equation}
\end{remark}

We close this subsection by giving an alternative formulation for (D1) that will be used in Section \ref{s:4}. 

\begin{remark}
	\label{rem:d1} 
	Consider an eigenvalue $\mu$ of $\check A=(A^0)^{-\frac12}A(A^0)^{-\frac12}$ (of multiplicity $\alpha$). Then, for $S$ as in (D1), let $\tilde R_\mu$ be an orthogonal right projector on the eigenspace of $\mu$ with respect to the matrix $S^{\frac12}\check AS^{-\frac12}$. Then (D1) implies that the real part of
	$$\tilde R_\mu^tS^{\frac12}(A^0)^{-\frac12}(-B+\mu^2\mathcal A -\mu C)(A^0)^{-\frac12}S^{-\frac12}\tilde R_\mu$$
	is negative. Clearly, $R_\mu=S^{-\frac12}\tilde R_\mu$ is a right projector on the eigenspace of $\mu$ with respect to $\check A$ and $\tilde R_\mu^tS^{\frac12}$ is the corresponding left-projector. \\
	In conclusion, (D1) implies that for each eigenvector $\mu$ of $\check A$ there exists a right projector $R_\mu$ such that the real part of
	$$L_\mu(A^0)^{-\frac12}(-B+\mu^2\mathcal A-\mu C)(A^0)^{-\frac12}R_\mu$$
	is negative. Additionally, if $(A^0)^{-1}A(\omega)$ (and thus also $\check A(\omega)$) is hyperbolic with constant multiplicities, the projector can locally be chosen smoothly with respect to $\omega \in S^{d-1}$.
\end{remark}

\subsection{Assumptions on the profile.}	Fix a planar shock solution $u(t,x)=\bar u(x_1)$ of \eqref{hypreg} connecting
endstates $u_\pm$ of a Lax-type shock solution of \eqref{inv}:
\begin{equation*} 
	\hbox{\rm
		$|\partial_x^k \bar u(x) - u_\pm| \leq e^{-\delta |x|}$ for $x\gtrless 0$, $0\leq k\leq K$,}
\end{equation*}
$K$ sufficiently large.
In particular, $\bar{u}$ satisfies the travelling wave ODE
\begin{equation} 
	\label{tw}
	B^{11}(\bar{u})\bar{u}'=f^1(\bar{u})+B^{11}(u_-)u_--f^1(u_-).
\end{equation}
In the following, we write $A_{\pm}(\xi)$ for $A(u_\pm,\xi)$, $B_{\pm}(\xi)$ for $B(u_{\pm},\xi)$ etc..

We impose the following standard assumptions on the profiles, which are in particular satisfied by shock waves in the relativistic dynamics of viscous, heat-conductive barotropic fluids \cite{FT14,FT18,BDN18,BDN19}.
\begin{enumerate}
	\item[(S1)] $(u_-,u_+,0)$ is a Lax-shock w.r.t. the original non-viscous system
	$$\sum_{j=0}^d (f^j(u))_{x_j}=0,$$
	i.e.,
	$\det(A^1_\pm)\neq 0$ and $A^1_+$ (resp. $A^1_-$) has $i_+$ negative (resp. $i_-$ positive) eigenvalues with
	$$i_++i_-=n+1.$$
	\item[(S2)] 
	The family $((A^0_\pm)^{-1}A_\pm(\omega))_{\omega \in S^{d-1}}$ is hyperbolic with constant multiplicities.
	\item[(S3)] 
	For all $x \in \R$, $\omega \in S^{d-1}$,
	$$\det B(\bar u(x),\omega) \neq 0$$
	and $A^1_\pm (B^{11}_\pm)^{-1}$ has no purely imaginary eigenvalue.
	\item[(S4)]
	$u_{\pm}$ are stable states of \eqref{hypreg} in the sense of Definition \ref{def:stable}.
	
\end{enumerate}

To obtain the optimal decay rate $(1+t)^{-(d-1)/4}$, we need an additional assumption on the so called glancing set of the first order operator (cf. \cite{GMWZ05,K11}). To this end, set 
$$p_{\pm}(\tau,\xi_1,\eta):=\det(\tau A^0_{\pm}+\xi_1 A^1_\pm+\sum_{j=2}^d\eta_j A^j_{\pm}).$$
By (S2),  $p_\pm(\cdot,\xi_1,\eta)$ has $1 \le \bar l_\pm \le n$ locally analytic roots
$\tau=a_l^{\pm}(\xi_1,\eta)$, $1 \le l \le \bar l_{\pm}$. We define the glancing set $\mG$ as the set of all $(\tau,\eta) \in \R \times \R^{d-1}\setminus\{0\}$ such that, for least one choice of sign and some $1 \le l \le \bar l_{\pm}$, $\tau-a_l^\pm(\cdot,\eta)$ has a root $\xi_1$ with multiplicity equal to or larger than $2$, i.e., suppressing the $\pm$ superscript, $(\underline\tau,\underline\eta) \in \mG$ if and only if there exists $\underline \xi_1 \in \R$, $1 \le l \le \bar l$, such that 
$$
	\underline \tau=a_l(\underline \xi_1,\underline \eta)~ \text{ and } ~ \partial_{\xi_1} a_l(\underline\xi_1,\underline\eta)=0.
$$
As $\underline \tau-a_l(\cdot,\underline \eta)$ is an analytic function, there exists $\bar s \ge 2$ with
\begin{equation}
	\label{bars}
	\partial_{\xi_1}^{s} a_l(\underline\xi_1,\underline\eta)=0~ \text{ for } 1 \le s \le \bar s-1,~~ \partial_{\xi_1}^{\bar s} a_l(\underline \xi_1, \underline \eta) \neq 0.
\end{equation}
Clearly, the implicit function theorem implies, locally, the existence of a smooth function $\tilde \xi_1$ with $\tilde \xi_1(\underline \eta)=\underline \xi_1$ such that
$$ \partial_{\xi_1}^{\bar s-1} a_l(\xi_1, \eta)=0 ~\Leftrightarrow~ \xi_1=\tilde \xi_1(\eta).$$
We now make the assumption that $\tilde \xi_1(\eta)$ persists as a root of multiplicity $\bar s$ of $\tau-a_l(\cdot,\eta)$.
\begin{enumerate}
	\item[(S5)] For all $(\underline \tau,\underline \eta) \in \mG$ with $\underline \xi_1$, $a_l$, $\tilde \xi_1$ as above,
	$$ \partial_{\xi_1}^s a_l(\tilde\xi_1(\eta),\eta)=0,~~1 \le s \le \bar s-1$$
	for $\eta$ close to $\underline \eta$. 
\end{enumerate}

By compactness (cf. \cite{GMWZ05}), there exist smooth functions $ \xi_{11},\ldots,\xi_{1\bar r}$, such that $\mG$ is the union of the finitely many smooth surfaces 
$$(\tau,\eta)=(a_{lr}(\eta),\eta):=\big( a_l^\pm(\xi_{1r}(\eta),\eta),\eta\big),~~ 1\le l \le \bar l, 1\le r \le \bar r,$$
on which the root of  $\tau-a_l^\pm(\cdot,\eta)$ is $\xi_{1r}(\eta)$ and has multiplicity $s_r \ge 2$.
\begin{remark}
	\begin{enumerate} 
		\item[(i)]
		As already stated  in the introduction, (S5) is automatically satisfied in dimension $d=2$ (\cite{GMWZ05}). In Section \ref{s:5.2}, we show that the arguments of Nguyen \cite{N09}. can be applied to the present situation, meaning we can remove (S5) at the expense of $(1+t)^{\frac14}$ decay rate.
		\item[(ii)]
		By the implicit function theorem, $(\underline \tau, \underline \eta) \in \mG$ with $\underline \tau=a_l(\underline \xi_1,\underline \eta)$ automatically satisfy (S5) if $\bar s=2$ in \eqref{bars}.
		In particular (S5) is satisfied for the relativistic barotropic Euler equations (which constitute of course the first order part for the models of relativistic viscous fluid dynamics) as here the  symbol splits as
		$$
			p_{\pm}(\tau,\xi_1,\eta)=p^1_\pm(\tau,\xi_1,\eta)p^2_\pm(\tau,\xi_1,\eta)
		$$
		with $p^1, p^2$ being polynomials of degree 2 (with respect to all variables $\tau,\xi_1,\eta$) and $(p^{1})^{-1}_\pm(\{0\}) \cap (p^2_\pm)^{-1}(\{0\})=\{(0,0,0)\}$.
	\end{enumerate}
\end{remark}

To avoid technical difficulties in the proof of Proposition \ref{prop:dec}, we impose the following condition.
\begin{enumerate}
	\item[(S6)]  There exist symmetrizer $S_\pm(u,\omega)$ for $A^0(u)^{-\frac12}A(u,\omega)A^0(u)^{-\frac12}$, such that for all eigenspaces $E=J_E^{-1}(\C^n)$ of $W_0$, $J_EW_1J_E^*$ is symmetric. $W_0$, $W_1$ being defined in condition (D1) for $u_*=u_\pm$. 
\end{enumerate}

\begin{remark}
	(S6) should be viewed as a purely technical condition, which is easily verified in all important applications. It is trivially satisfied if $A^j,B^{jk}$, $j,k=0,\ldots,d$, are simultaneously symmetrizable, as is the case in the models of relativistic barotropic viscous fluid dynamics in \cite{FT14,FT18,BDN18} or if all eigenvalues of $J_E^*W_1J_E$ are semi-simple with constant multiplicities, which is satisfied by the models \cite{BDN19}; cf. \cite{F20} for the symmetric structure of the models in \cite{BDN18} and \cite{FS25} for the statement concerning \cite{BDN19}.
\end{remark}

\subsection{Consistent splitting and Evans function.}  In order to formulate the Evans function conditions, we study the inhomogeneous Cauchy problem associated to the linearization of \eqref{hypreg} about the profile $\bar u$ given by

\begin{align}
	\label{line1}
	&\mathcal A(\bar u)u_{tt}-\sum_{j=1}^d(B^{jk}(\bar u)u_{x_j})_{x_k}-\sum_{j=1}^d(C^j(\bar u) u_t)_{x_j} + \tilde A^0(\bar u)u_t + \sum_{j=1}^d(\tilde A^j(\bar u)u)_{x_j}=f,\\
	\label{liniv}
	&u|_{t=0}=\phi-\bar u,\quad \partial_t u|_{t=0}=\psi,
\end{align}
with $\phi,\psi,f: \R^d \to \R^n$ sufficiently smooth, and
\begin{align*}
	(\tilde A^0(\bar u))_{lm}&=(A^0(\bar{u}))_{lm}-\sum_{q=1}^n(\partial_m(C^{1}_1)_{lq}(\bar u)-\partial_q(C^1_1)_{lm})(\bar u')^q,\\
	(\tilde A^j(\bar u))_{lm}&=(A^j(\bar{u}))_{lm}-\sum_{q=1}^n
	\partial_mB^{1j}_{lq}(\bar u)(\bar u')^q,~~j=1,\ldots,d.
\end{align*}
For notational purposes, we define the symbols 
\begin{align*}
	B_{12}(u,\eta)&=\sum_{j=2}^d \eta_jB^{j1}(u),\quad B_{22}(u,\eta)=\sum_{j,k=2}^dB^{jk}(u)\eta_j\eta_k\\
	C_{2}(u,\eta)&=\sum_{j=2}^dC^j(u)\eta_j,\quad A_2(u,\eta)=\sum_{j=2}^dA^j(u)\eta_j,~~(u,\eta) \in \R^n \times \R^{d-1}
\end{align*}
($\tilde A_2$ analogously), and write $(x,y_1,\ldots,y_{d-1})$ instead of $x=(x_1,\ldots,x_d)$.
Then applying Laplace-transform with respect to $t$ and Fourier-transform with respect to $y$ in \eqref{line1}, gives
\begin{equation} 
	\label{flt}
	(B^{11}(\bar{u})\hat u_{x})_x+(S(\bar{u},\lambda,\eta) \hat u)_{x}-(s(\bar{u},\lambda,\eta)\hat u=\hat f,
\end{equation}
where $\lambda=i\tau+\gamma \in \C$, $\eta \in \R^{d-1}$ are the dual variables to $t$, $y$ and
\begin{align*}
	S(\bar{u},\lambda,\eta)&=iB_{12}(\bar{u},\eta)+\lambda C^1(\bar{u})-\tilde A^1(\bar{u}),\\
	s(\bar{u},\lambda,\eta)&=\lambda^2\mathcal A-i\lambda C_2(\bar{u},\eta)+B_{22}(\bar{u},\eta)\\
	&\quad+\lambda \tilde A^0(\bar{u})+i\tilde A_2(\bar{u},\eta).
\end{align*}
In variables $(\hat u,\hat u_x)$, \eqref{flt} corresponds to the  first-order system of differential equations
$$\begin{pmatrix}
	\hat u \\ \hat v
\end{pmatrix}_x-\begin{pmatrix}
	0 & B^{11}(\bar{u})^{-1}\\
	s(\bar{u},\eta,\lambda)-S(\bar u,\lambda,\eta)_x & -S(\bar u,\lambda,\eta)B^{11}(\bar u)^{-1}\end{pmatrix}\begin{pmatrix}
	\hat u \\ 
	\hat v\end{pmatrix}=\begin{pmatrix}
	0 \\ \hat f
\end{pmatrix},$$
which we abbreviate as
\begin{equation} 
	\label{lin2}
	V_x(x)-G(x,\zeta)  V(x)=  \hat F(x).
\end{equation}

We introduce $\zeta=(\tau,\gamma,\eta)\in \R \times \R \times \R^{d-1}\equiv \C \times \R^{d-1}$, and use polar coordinates
$$\zeta=\rho\hat\zeta, ~~\rho \ge 0, \hat \zeta \in S^d.$$ With slight abuse of notation we also write $\zeta=(\lambda,\eta) \in \C \times \R^{d-1}$. 

In order to show the existence of an Evans function for \eqref{lin2}, we need to analyse the behaviour of the coefficient matrices at the end states $G_\pm(\zeta)$. To this end, note that due to \eqref{asym},  $\tilde A^j_\pm \equiv A^j_\pm$.

We use the following notion \cite{AGJ90,GZ98}.  

\begin{defi}
	An open subset $\Omega$ of $\C \times \R^{d-1}$ is called a domain of consistent splitting for the shock profile $\bar u$ if for any $\zeta \in \Omega$, the matrices  $G_\pm(\zeta)$ have no purely imaginary eigenvalue and the respective dimension of the stable (and thus unstable) subspace $S_\pm(\zeta)$ ($U_\pm(\zeta)$)  of  $G_+(\zeta)$, $G_{-}(\zeta)$ coincide  (and are then necessarily constant for $\zeta \in \Omega$).
\end{defi}

Next, for $c>0$, define
$$M_c=\{\zeta \in \R^{d+1}: \gamma \ge -c\kappa(|(\eta,\tau)|)\}\setminus\{0\},~~\kappa(\rho)=\frac{\rho^2}{1+\rho^2},$$
and denote by
$$S^d_c=\{(\rho,\hat \zeta) \in (0,\infty) \times S^d:\hat \gamma \ge -c\kappa(\sqrt{\rho}|(\hat \tau,\hat \eta)|)\}$$
the corresponding set in polar coordinates. We also set $S^{d}_+=S^d \cap \{\hat \gamma \ge 0\}$. 
\begin{prop}
	\label{prop:spec1}
	\begin{enumerate}
		\item[(i)] There exists, $c>0$ such that  $M_c$ is a domain of consistent splitting for $\bar u$.
		\item[(ii)] $0$ is a semi-simple eigenvalue of multiplicity $n$ of $G_\pm(0,0)$. The other eigenvalues (fast modes) are those of $-
		A^1_\pm(B^{11}_\pm)^{-1}$.
		\item[(iii)] Let $\hat \gamma\ge c_1>0$ for some positive constant $c_1$. Then, an eigenvalue $\mu$ of $G_\pm(0,\hat \zeta)$ splitting from $0$ admits the expansion
		$$\mu(\rho,\hat \zeta)=\rho\mu_1(\hat \zeta)+O(\rho), ~~\rho \to 0,$$
		where $\mu_1$ is an eigenvalue of 
		$$-(A^1_\pm)^{-1}(\hat \lambda A^0_\pm+i A_{2\pm}(\hat \eta)).$$
		In particular, $\Rep \mu_1 \neq 0$.
	\end{enumerate}
	
\end{prop}

\begin{proof}
	
	(i) We first show that for some $c>0$, $G_+(\zeta)$ has no purely imaginary eigenvalues if $\zeta \in M_c$.  To this end, suppose that for all $c>0$ there exists $\zeta \in M_c$ and an eigenvalue $\mu=i\xi_1$, $\xi_1 \in \R$ of $G_+(\zeta)$. For a corresponding eigenvector $V=(v_1,v_2) \in \R^n  \times \R^n$ and $\xi=(\xi_1,\eta) \in \R \times \R^{d-1}$,  we find 
	$$(\lambda^2 \mathcal A+B(\xi)-i\lambda C(\xi)+\lambda A^0+iA(\xi))v_1=0.$$
	But due to \eqref{imp} and \eqref{rep}, each solution $(\xi,\lambda)=(\xi_1, \eta, \gamma,\tau)$ of \eqref{disp} satisfies
	\begin{equation}
		\label{gamma}
		\hat \gamma \ge -\tilde c\kappa(|(\eta, \tau)|).
	\end{equation}
	for some $\tilde c>0$, which is a contradiction to our assumption. The same argumentation also holds for $G_-$. Due to continuity with respect to $x$, $\dim U_+=\dim U_-$ is shown if for one $\zeta \in M_c$, $G(x,\zeta)$ has no purely imaginary for all $x \in \R$. To this end, note that for $\tau=\eta=0$ and $\gamma \to \infty$ the eigenvalues $\mu$ of $G(x,0,\gamma,0)$ converge to those of (suppressing the argument $\bar u$)
	$$\begin{pmatrix}
		0 & (B^{11})^{-1}\\ \mathcal A &  -C^1(B^{11})^{-1}
	\end{pmatrix},$$
	which are the solutions of
	$$\det(-\mathcal A+\mu^2 B^{11}+\mu C^1) =0.$$
	In particular, for $\mu=i\nu$, $\nu \in \R$,
	$$\det (-\mathcal A-\nu^2 B^{11}+i\nu C^1) = 0.$$
	This means that $\lambda=1$ is an eigenvalue of $\mB(\bar u,\nu,0_{\R^{d-1}})$, a contradiction to (H$_B$). Thus, the dimension of the stable space of $G(x,\gamma,0,0)$ is indeed constant.
	Assertion (ii) follows directly from $s(\bar u,0)=0$, $S(\bar u, 0)=-A^1(\bar u)$.
	
	(iii) Let $\gamma=\Rep \lambda \ge c_1>0$. If $\mu$ is an eigenvalue of $G_+$, then, for some $v\in \C^n$,
	\begin{equation}
		\label{mu0}
	\left[\mu^2B^{11}+\rho \mu(\hat \lambda C^1+i B_{12}(\hat \eta))+\rho^2B'(\hat \lambda,\hat \eta)-\rho (\hat \lambda A^0+iA_2(\hat \eta))-\mu A^1\right]v=0.
	\end{equation}
	Considering the expansions $\mu=\mu_1\rho+O(\rho)$, $v=v_0+O(\rho^2)$ gives
	$$(\mu_1A^1+\hat \lambda A^0(u)+i A_2(u,\hat \eta))v_0=0.$$
	Now suppose $\mu_1=i\xi_1$ for $\xi_1 \in \R$.   Then $\det(\hat\lambda +i(A^0)^{-1}A(\xi_1,\hat \eta))=0$ for $(\xi_1,\hat \eta) \in \R^d$ and $\lambda \in \C$ with $\Rep \lambda \neq 0$, which contradicts (H$_A$). 
\end{proof}

From now on let $N$ ($1 \le N \le 2n)$ denote the dimension of $S_{\pm}$, the stable space of $G_\pm$ and $M=2n-N$ the dimension of $U_{\pm}$, the unstable space of $G_\pm$. For  ODEs of the form \eqref{lin2} satisfying properties (i), (ii), (iii) of Proposition \ref{prop:spec1}, one can conlcude the following statement via the Gap Lemma \cite{GZ98,Z01,MZ05}; cf. e.g. \cite{Z01, Z04, KwZ09, Kw11} for a detailed construction of the Evans function.

\begin{prop}
	\label{prop:ev}
	Let $c>0$ be sufficiently small. Then for  $\zeta \in M_c\setminus\{0\}$, there exist bases of solutions
	$$\{V_1^+(x,\zeta),\ldots,V_N^+(x,\zeta)\}~/~\{V_1^-(x,\zeta),\ldots,V_{M}^-(x,\zeta)\}$$
	to \eqref{lin2} with $\hat F=0$, spanning the stable/unstable manifold of $G(\bar u(x),\zeta)$ at $x=+\infty/-\infty$ such that the Evans function
	$$\mathcal D(\zeta)=\det(U_1^+,\ldots,U_n^+,U_1^-,\ldots,U_n^-)|_{x=0} $$ 
	is analytic on $M_c$ and continuously extendible to $\overline {M_c}$.
	
	In addition, $D(\rho,\hat \zeta)=\mD(\rho\hat\zeta)$ is analytic on $S^d_c$, and $D$, $\partial_\rho D$ extend continuously to $\overline{S^d_c}$. In particular,
	$$D \in C(S^d_c,C^1([0,\infty)).$$
\end{prop}

Now, the spectral stability of the shock profile corresponds to the following assumption.
\begin{itemize}
	\item[(S7)] $D$ has no zeros on $S^d_c$ and $\partial_\rho D(0,\hat \zeta) \neq 0.$
\end{itemize}

\begin{remark}
	\label{rem:phi}
	Differentiating \eqref{tw} shows that $\bar v=\bar{u}'$ is a solution to \eqref{flt} for $\hat f=0$ and $\rho=0$ with $\bar v(\pm \infty)=0$. In particular, for $\rho=0$, $(\bar v,0)$ solves \eqref{lin2} with $\hat F=0$, and vanishes at $\pm \infty$, i.e., for $\zeta=0$,
	$$
		(\bar v,0)^t \in \ker \big(\operatorname{span}\{V_1^+,\ldots,V_N^+\}\big) \cap \ker \big(\operatorname{span}\{U_1^-,\ldots,U_M^-\}\big).
	$$
	This shows that $D(\rho,\hat \zeta)$ always vanishes for $\rho=0$.
	
\end{remark}

\section{Linearized stability}\label{s:3}
Following \cite{Z07}, we proof linearized stability by virtue of resolvent estimates for the differential operator associated with \eqref{line1}. Concretely, we write \eqref{line1}, \eqref{liniv} in variables $U=(u, \mathcal Au_t)$, to obtain (suppressing the argument $\bar{u}$ throughout)
\begin{align}
	\label{fo}
	&\begin{aligned}
		U^1_t&=\mathcal A^{-1}
		U^2\\
		U^2_t&=\sum_{j,k=1}^dB^{jk}
		U^1_{x_jx_k}-\sum_{j=1}^d\tilde A^j U^1_{x_j}+\sum_{j=1}^dC^j\mathcal A^{-1}U^2_{x_j}-\tilde A^0\mathcal A^{-1}U^2\\
		&\quad + \sum_{j=1}^dB^{j1}_{x_1}U^1_{x_j}-\tilde A^1_{x_1}U^1+(C^1\mathcal A^{-1})_{x_1}U^2+f
	\end{aligned}\\
	\label{foiv}
	& U|_{t=0}=U_0:=\begin{pmatrix} \phi-\bar u \\ \mathcal A^{-1}\psi\end{pmatrix}
\end{align}
In form of general first-order in time evolution equation \eqref{fo} is
\begin{equation} 
	\label{semi1}
	U_t-LU=\begin{pmatrix}
		0_{\R^n} \\ f
	\end{pmatrix}, 
\end{equation}
where
\begin{align}
	\label{l}
	L&=\begin{pmatrix}
		0 &  \mathcal A^{-1} \\
		L^{21} & L^{22}
	\end{pmatrix},\\
	\nonumber
	L^{21}u&=\sum_{j,k=1}^d(B^{jk}u_{x_j})_{x_k}-\sum_{j=1}^d( \tilde A^ju)_{x_j},\quad L^{22}u= \sum_{j=1}^d (C^j\mathcal A^{-1}u)_{x_j}-\tilde A^0\mathcal A^{-1}u,
\end{align}
For $s \in \N_0$, the $H^s$-realisation of $L_s$ of $L$ is the (unbounded) operator on $H^{s+1} \times H^s$ defined by
$D(L_s)=H^{s+2} \times H^{s+1}, \quad L_sF =LF,~~F \in D(L_s)$. 

We also introduce the equivalent operator $\tilde L_s$ on $H^s \times H^s$ via
\begin{equation} 
	\label{tl}
	D(\tilde L_s)=H^{s+1} \times H^{s+1}, ~~\tilde L_s=
	\begin{pmatrix}
		\Lambda& 0\\
		0 & I_{L^2}
	\end{pmatrix}L
	\begin{pmatrix}
		\Lambda^{-1} & 0\\
		0 & I_{L^2}
	\end{pmatrix}=:TLT^{-1}.
\end{equation}
Since $T$ is an isometric linear operator from $H^{s+1} \times H^s$ to $H^{s} \times H^s$, we find for $F=(F^1,F^2) \in H^{s+2} \times H^{s+1}$,
$$\|LF\|_{s+1,s}=\|\tilde L_s TF\|_{s,s}.$$
Now, let
$$\tilde{\mathbb{L}}(\bar u,\xi)=\begin{pmatrix}
	0 & \lxi \mathcal A(\bar u)^{-1} \\
	B(\bar u,\xi)\lxi^{-1} & iC(\bar u,\xi)\mathcal A(\bar u)^{-1}
\end{pmatrix}$$
be the principal symbol of $\tilde{L}$. By (H$_B$), there exists a symbolic symmetrizer for 
$$i\mathcal B(\bar u,\omega)=i\begin{pmatrix} 0 & I\\ \mathcal A^{-1}(\bar u)B(\bar u,\omega) & i\mathcal A(\bar u)^{-1}(\bar u)C(\bar u,\omega)
\end{pmatrix}=i\begin{pmatrix}I & 0\\ 0& \mathcal  A(\bar u)^{-1}\end{pmatrix}\tilde {\mathbb L}(\bar u,\xi)\begin{pmatrix}I & 0\\ 0& \mathcal  A(\bar u)\end{pmatrix},$$
and thus also for $\tilde{\mathbb{L}}(\bar u,\xi)$. Hence, there exists a functional symmetrizer for $\tilde L$, and we obtain the following proposition \cite{BS06}.

\begin{prop}
	\label{semi2}
	For all $s \in \N_0$, $T \in [0,\infty)$ and $U_0 \in H^{s+1}$, there exists a unique solution $U \in C([0,T],H^{s+1}) \cap C^{1}([0,\infty],H^{s})$ of the Cauchy problem
	$$	
		U_t-\tilde L_sU=0,\quad U(0)=U_0.
	$$
	Additionally,
	$$
		\|U\|_s^2 \le C(T)\|U_0\|_s^2 .
	$$
\end{prop}

It is straightforward to show, note in particular $\det B(\bar u,\omega) \neq0$, that $\tilde L_s$ is a closed operator, and trivially $\tilde L_s$ is densely defined. Thus, from Proposition \ref{semi2} we derive the following corollary (cf. e.g. \cite[Theorem 4.1.3]{PA83}).
\begin{coro}
	\label{coro:semi}
	For all $s \in \N_0$, $\tilde L_s$ generates a $C^0$-semigroup $e^{\tilde  L_st}$ on $H^s$ with
	\begin{equation} 
		\label{sg}
		\|e^{\tilde L_st}\|_s \le C_se^{\gamma_st}.
	\end{equation}
\end{coro}
By definition, $L_s$ generates the $C_0$-semigroup $e^{L_st}=Te^{\tilde L_st}T^{-1}$ on $H^{s+1} \times H^s$ satisfying
$$\|e^{L_st}\|_{s+1,s} \le C_se^{\gamma_st}$$
with the same constants as in \eqref{sg}. 

In the following, we mostly consider $s=0$ and write $e^{Lt}$ instead of $e^{L_0t}$. As stated above, the crucial part of our argumentation consists in estimates for the resolvent equation associated to \eqref{semi1}
\begin{equation}
	\label{flt1}
	(\lambda -L)U=F,~~\lambda \in \C,
\end{equation}
with general right hand side $F=(F^1,F^2)$, $F^1, F^2: \R^d \to \R^n$. To this end, we Fourier-transform  \eqref{flt1} with respect to $y=(x_2,\ldots,x_d)$ and consider the resulting family of ordinary differential equations (from now on we often write $x \in \R$ for $x_1$)
\begin{equation} 
	\label{flt2}
	(\lambda -L_\eta)\hat U(x,\eta) =\hat F(x,\eta), ~~ (\lambda,x,\eta) \in \C \times \R \times \R^{d-1},
\end{equation} 
where $\hat U, \hat F$ denote the Fourier-transformed versions of $U,F$ with respect to $y$ and
\begin{align*}
	L_\eta&=\begin{pmatrix}
		0 & \mathcal A^{-1} \\
		L^{21}_\eta & L^{22}_\eta
	\end{pmatrix},\\
	L^{21}_\eta u&=(B^{11}u_x)_x+((iB_{12}(\eta)-A^1)u)_x-(B_{22}(\eta)- i\tilde A_2(\eta)u,\\
	L^{22}_\eta u&=(C^1\mathcal A^{-1}u)_x+(iC_{2}(\eta)-\tilde A^0)\mathcal A^{-1}u.
\end{align*}

At the heart of the argument lie the following estimates, the proofs of which are shifted to Section \ref{s:4}.
\begin{prop}
	\label{prop:resolvent}
	\begin{itemize}
		\item[(i)]
		 There exists $r,c>0$  such that for all $(\lambda,\eta) \in M_c$ with $\rho=|(\lambda,\eta)|\le r$ and for all $p \in [2,\infty]$, $k=0,1$,
		\begin{align} 
			\label{sfre}
			|(\lambda-L_\eta)^{-1}\partial_{x_1}^k \hat F|_{L^p_{x_1}} &\le C\beta \rho^{-1+k}(|\hat F^1|_{W^{1,1+k}_{x_1}}+|\hat F^2|_{L^1_{x_1}})
		\end{align}
		with 
		$$
			\beta=\max_{\substack{l=1,\ldots,\bar l\\r=0,\ldots, \bar r}}(|\operatorname{Im}\hat \lambda-a_{lr}(\hat \eta)|+\operatorname{Re}\hat \lambda +\rho)^{1/\bar s_{r}-1}, ~~\hat \lambda=\lambda/\rho,~\hat \eta=\eta/\rho,
		$$
		$s_0=a_{l0}=1$, and otherwise $a_{lj}$, $s_j$ are defined below condition (S5).
		\item[(ii)]
		For all $r_1,r_2>0$, $s \in \N$ and
		$(\lambda,\eta) \in M_c$ with $r_1 \le |(\lambda,\eta)| \le  r_2,$ 
		there holds
		$$|(\lambda-L_\eta)^{-1}\hat F|_{H^s_{x_1} \times H^s_{x_1}} \le C(r_1,r_2,s)|\hat F|_{H^{s}_{x_1}\times H^{s-1}_{x_1}}.$$
		\item[(iii)]
		For $R>0$, define the cut-off
		$$\chi_R(\lambda,\eta)=\begin{cases}
			1, & |\eta|^2 \ge 2(R-|\lambda|^2)\\
			0, & |\eta|^2 \le R -|\lambda|^2
		\end{cases},$$
		and set $\chi_R(\lambda,D)=\mF^{-1}_\eta\chi_R\mF_y$. Then there exists $\delta_1>0$ such that, if $\|\bar u-\bar u_-\|_{W^{2,\infty}} <\delta_1$, for all $s \in \N$ and some $R=R(s),c=c(s)>0$, we find
		\begin{equation} 
			\label{lfre}
		\|(\lambda-L)^{-1}\chi_R(\lambda,D)F\|_{s+1,s} \le C(s)\|F\|_{s+1,s},
	\end{equation}
 	for all $\lambda \in \C$ with $\Rep \lambda \ge-c$. 
	\end{itemize} 
\end{prop}

From here, we can derive estimates for $e^{L_st}$ via the identity (cf. \cite{PA83})
\begin{equation} 
	\label{sg1}
	(e^{Lt}F)(x)=\text{~P.V.~}\int_{\gamma-i\infty}^{\gamma+i\infty} \int_{\R^{d-1}}e^{i\eta y+\lambda t}(\lambda-L_\eta)^{-1}\hat F(x_1,\eta)d\eta d \lambda,
\end{equation}
for $F=(F^1,F^2) \in D(L)=H^{2} \times H^{1} $, $t \ge 0$,  $\gamma>\gamma_0$ (the growth-rate of $e^{L_st}$), with convergence of the integrals in $L^2([0,T] \times \R^d)$. More precisely, we use the following result leading to a decomposition of $e^{Lt}$ into a low and high frequency part. We can directly apply the proof of \cite[Corollary 4.11]{Z07} with $(\Lambda F^1, F^2) \in H^1$, playing the role of $f$. Note that in \cite{Z07} one also has access to the pointwise estimate
	$$
		|(1+|\eta|)\chi_R(\lambda,\eta)(\lambda-L_\eta)^{-1}\hat f|_{H^1_{x_1}} \le C|(1+|\eta|)\hat f|_{H^1_{x_1}}.
	$$
	However, only the resulting integrated estimate
	$$	
		\|(\lambda-L)^{-1}\chi_R(\lambda,D)f\|_{1} \le C\|f\|_{1}
	$$
	is used in the proof of Corollary 4.11.
	
\begin{coro}
	\label{coro:decomp}
		Let the operator $L$ satisfy Proposition \ref{prop:resolvent} (ii), (iii). Then, for arbitrary $\theta_2 >0$ and $\theta_1$ sufficiently small in relation to $\theta_2$ and all $t \in [0,\infty)$, $e^{Lt}=:S(t)=S_1(t)+S_2(t)$ with
	\begin{equation}
		\label{s1s2}
		\begin{split}
		(S_1(t)F)(x,y)&=\int_{|\eta|^2 \le \theta_1+\theta_2} \oint_{\Rep\lambda=\theta_2-|\eta|^2-|\operatorname{Im}\lambda|^2\ge-\theta_1} e^{i\eta \cdot y+\lambda t} (\lambda-L_\eta)^{-1}\hat F(x,\eta)d\lambda d\eta\\
		(S_2(t)F)(x,y)&=\mathrm{P.V.}\int_{-\theta_1-i\infty}^{-\theta_1+i\infty}\int_{\R^{d-1}} \mathbb I_{\{|\eta|^2+|\operatorname{Im}\lambda|^2 \ge \theta_1+\theta_2\}} e^{i \eta \cdot y+\lambda t}(\lambda-L_\eta)^{-1}\hat F(x,\eta)d\eta d\lambda,
		\end{split}
	\end{equation}
	for $F\in D(L)=H^2 \times H^1$. $\mathbb I_A$ being the indicator function of a set $A$.
\end{coro}

Using an argumentation similar to \cite[Section 13.1]{GMWZ05}, \cite[Proof of Proposition 4.5]{Z07}, we can deduce the following semi-group estimates.
\begin{prop}
	\label{prop:dl1}
	There exist $\theta_1,\theta_2$ sufficiently small such that for $S_1,S_2$ defined in \eqref{s1s2}, the following holds:
	\begin{enumerate} 
		\item[(i)]
	For $p \in [2, \infty]$, $t\in [0,\infty)$, $\kappa \in \N_0^d$, $|\kappa| \le 1$, 
		\begin{equation} 
			\label{s1}
			\|S_1(t)\partial_{x}^\kappa F\|_{L^p} \le C(1+t)^{-\frac{d-1}{2}(1-\frac{1}{p})-\frac{|\kappa|}{2}}(\|F^1\|_{W^{1+|\kappa|,1}}+\|F^2\|_{L^1}).
		\end{equation}
		\item[(ii)]
		For $s\ \in \N$, $t\in [0,\infty)$ 
		\begin{equation} 
			\label{lf}
			\|S_2(t)F\|_{s+1,s} \le Ce^{-ct}(\|F\|_{s+1,s}+\|LF\|_{s+1,s}+\|L^2F\|_{s+1,s}).
		\end{equation}

		\item[(iii)]
		For $t \in [0,\infty)$, $p \in [2,\infty]$, 
		$$
		\|S_2(t)F\|_{L^p} \le Ce^{-ct}\|F\|_{ s_0+3,s_0+2},~~s_0:=[d/2]+1.
		$$
	\end{enumerate}
	The above holds for $F=(F^1,F^2):\R^d \to \R^n \times \R^n$ as regular as required for the right hand sides to make sense.
\end{prop}

\begin{proof}
	\textit{Proposition \ref{prop:resolvent} (i) $\Rightarrow$ Proposition \ref{prop:dl1} (i).}  By \eqref{sfre} and Hausdorff-Young inequality,  we can estimate for $(\lambda,\eta) \in M_c$ $\rho=(|\eta|+|\lambda|)<r$, $k=0,1$, $p \in [2,\infty]$,
	\begin{align*}
		\|(\lambda-L_\eta)^{-1}\partial_{x_1}^k\hat F\|_{L^\infty_\eta L^p_{x_1}} &\le C\beta\rho^{-1+k}(\|\hat F^1\|_{L^\infty_\eta W^{1+k}_{x_1}}+\|\hat F^2\|_{L^\infty_\eta L^1_{x_1}}\\
		& \le C(\|F^1\|_{W^{1+k,1}}+\|F^2\|_{L^1})\beta\rho^{-1+k}.
	\end{align*}
	Then, the proof goes exactly as \cite[Proof of Proposition 13.1]{GMWZ05} in the small frequency regime or \cite[Proof of Proposition 4.13]{Z07} with $\|f\|_{L^1}$ replaced by $\|F^1\|_{W^{1,1}}+\|F^2\|_{L^1}$ or $\|F^1\|_{W^{2,1}}+\|F^2\|_{L^1}$ for derivative estimates.
	
	\textit{Proposition \ref{prop:resolvent} (ii)$+$(iii) $\Rightarrow$ Proposition \ref{prop:dl1} (ii).}  For $F=(F^1,F^2) \in H^{s+1} \times H^s$, set 
	$$
		\tilde F_\lambda=\mF^{-1}_\eta( \mathbb{I}_{\{|\eta|^2+|\operatorname{Im}\lambda|^2\ge \theta_1+\theta_1\}}\hat F).
	$$
	Then, Proposition \ref{prop:resolvent} (i), (ii) yield for all $s \in \N$ and some $C,c>0$,
	\begin{equation} 
		\label{decf}
		\|(\lambda-L)^{-1}\tilde F_\lambda\|_{s+1,s} \le C(\|F\|_{s+1,s}),~~\Rep \lambda \ge -c.
	\end{equation}
	Via the resolvent identity
    $$
    	(\lambda- L)^{-1}G =\lambda^{-2}\left((\lambda- L)^{-1} L^2G+ LG \right)+\lambda^{-1}G,~~ \lambda \in \rho(L),~ G \in D(L^2)=H^3 \times H^2,
	$$
	we find
	\begin{align*} 
	\|S_2(t) F\|_{s+1,s}&\le e^{-\theta_1 t}\int_{-\theta_1-i\infty}^{-\theta_1+i\infty} |\lambda|^{-2}\left(\|(\lambda I-L)^{-1}L^2\tilde F_\lambda\|_{s+1,s}+\|L\tilde F_{\lambda}\|_{s+1,s}\right)d\lambda\\
	&\quad +\left\|\mathrm{P.V.}\int_{-\theta_1-i\infty}^{-\theta_1+i\infty}e^{\lambda t}\lambda^{-1}\tilde F_\lambda d\lambda\right \|_{s+1,s}:=I_1+I_2
	\end{align*}
	Clearly, $\|L^k \tilde F\|_{s+1,s} \le \|L^k F\|_{s+1,s}$, $k=0,1,2$,
	and we get
	\begin{align*}
		I_1 &\le e^{-\theta_1 t}\int_{-\infty}^\infty (|\theta_1|^2+|\tau|^2)^{-1}d\tau (\|LF\|_{s+1,s}+\|L^2F\|_{s+1,s}) \\
		&\le Ce^{-\theta_1 t}(\|LF\|_{s+1,s}+\|L^2F\|_{s+1,s}).
	\end{align*}
	Since $\tilde F_\lambda=F$ for $|\operatorname{Im}\lambda|^2 \ge \theta_1+\theta_2$, we find
	\begin{align*}
			\mathrm{P.V.}\int_{-\theta_1-i\infty}^{-\theta_1+i\infty}e^{\lambda t}\lambda^{-1}\tilde F_\lambda d\lambda&=\int_{-\theta_1-i\sqrt{\theta_1+\theta_1}}^{-\theta_1+i\sqrt{\theta_1+\theta_2}}e^{\lambda t}\lambda^{-1}F_\lambda d\lambda\\
			&+\quad e^{-\theta_1 t}\int_\R\mathbb I_{|\tau|^2 \ge \theta_1+\theta_2}e^{i\tau}(\theta_1+i\tau)^{-1}d\tau F.
	\end{align*}
	The  last integral is finite due to the residue theorem, and hence,
	$$
		I_2 \le Ce^{-\theta_1 t}\|F\|_{s+1,s},
	$$
	which finishes the proof.
	
	\textit{Proposition \ref{prop:dl1} (iii).} Follows directly from (ii) via Sobolev embedding $H^{s_0+k} \hookrightarrow W^{k,\infty}$ and interpolation between $L^2$ and $L^\infty$.
\end{proof}

Different from the hyperbolic-parabolic case \cite{Z07} or relaxation case \cite{KwZ09}, \cite{Kw11}, we need an improved decay rate $(1+t)^{-(d-1)/4-1/2}$ not only for space-conservative but also for space-time-conservative sources. This is the purpose of the following result.

\begin{prop}
	\label{prop:dect}
	For $t \in (0,\infty)$, $f \in W^{1,1}([0,T], H^2) \cap L^1([0,T],L^2)$, $F=(0,f)$, there holds
	\begin{align*}
		\left\|\int_0^te^{L(t-s)}\partial_s F(s)ds\right\|_{1,0} &\le C(\|f(t)\|+(1+t)^{-\frac{d-1}{4}}\|f(0)\|_{ L^1\cap H^2})\\
		&\quad +C\int_0^te^{-c(t-s)}\|f(s)\|_{2
		}+(1+t-s)^{-\frac{d-1}{4}-\frac12}\| f(s)\|_{L^1}ds,
	\end{align*}
and for $p \in (2,\infty)$, $f$ additionally in $L^1([0,T], L^1 \cap H^{s_0+2}),$
	\begin{align*}
		\left\|\int_0^te^{L(t-s)}\partial_s F(s)ds\right\|_{L^p} &\le C(\|f(t)\|_{L^p}+(1+t)^{-\frac{d-1}{2}(1-\frac1p)}\|f(0)\|_{L^1 \cap H^{s_0+2}})\\
		&\quad +C\int_0^t(1+t-s)^{-\frac{d-1}{2}(1-\frac1p)-\frac12}\| f(s)\|_{L^1 \cap H^{s_0+3}}ds.
	\end{align*}	
\end{prop}

\begin{proof}
	By general results on semigroups, we have for $F\in W^{1,1}([0,T],H^1 \times L^2) \cap L^1([0,T],D(L))$
	\begin{align*} 
		\int_0^te^{L(t-s)}\partial_sF(s)ds&=F(t)-e^{Lt}F(0)+\int_0^te^{L(t-s)}LF(s)ds\\
		&=F(t)-e^{Lt}F(0)+\int_0^t(S_1(t-s)+S_2(t-s))LF(s)ds.
	\end{align*}
	Thus, due to Proposition \ref{prop:dl1} and interpolation,  it suffices to show
	\begin{align*} 
		\|S_1(t)LF\| &\le C(1+t)^{-\frac{d-1}{4}-\frac12}\|f\|_{L^1\cap L^2}, \\
		\|S_1(t)LF\|_{L^p} &\le C(1+t)^{-\frac{d-1}{2}(1-\frac1p)}\|f\|_{L^1 \cap H^{s_0}}.
	\end{align*}
	
 First, the resolvent equality
	$$(\lambda-L_\eta)^{-1}L_\eta\hat F=\lambda(\lambda-L_\eta)^{-1}\hat F-\hat F,~~ \hat F \in D(L_\eta),$$
	gives 
	\begin{align*}
		S_1(t)LF&=\int_{|\eta|^2 \le \theta_1+\theta_2}\oint_{\Rep \lambda=\theta_2-|\eta|^2-|\operatorname{Im}\lambda|^2 \ge-\theta_1} e^{i\eta \cdot y+\lambda t}(\lambda-L_\eta)^{-1}\lambda\hat F(x,\eta) d\lambda d\eta\\
		&\quad +\int_{|\eta|^2 \le \theta_1+\theta_2}\oint_{\Rep \lambda=\theta_2-|\eta|^2-|\operatorname{Im}\lambda|^2 \ge-\theta_1} e^{i\eta \cdot y+\lambda t}\hat F(x,\eta) d\lambda d\eta:=I_1+I_2.
	\end{align*}
	By Proposition \ref{prop:resolvent} (i),
	$$|(\lambda-L_\eta)^{-1}\lambda \hat F|_{L^p} \le C\beta |\hat f|_{L^1},~~p \in [2,\infty],$$
	for $(\lambda,\eta) \in M_c$, $|(\lambda,\eta)|$ sufficiently small.
	 Hence, arguing again as in \cite[Proof of Proposition 4.13]{Z07} , we find
	$$\|I_1\|_{L^p} \le C(1+t)^{-\frac{d-1}{4}(1-\frac1p)-\frac12}\|f\|_{L^1}.$$
	Due to the analyticity of $\lambda \to e^{\lambda t}$, we get for $|\eta|^2 \le \theta_1+\theta_2$
	\begin{align*}
		\oint_{\Rep \lambda=\theta_2-|\eta|^2-|\operatorname{Im}\lambda|^2 \ge-\theta_1} e^{\lambda t}d\lambda&=t^{-1}e^{-\theta_1 t}(e^{i\sqrt{\theta_1+\theta_2-|\eta|^2}t}-e^{-i\sqrt{\theta_1+\theta_2-|\eta|^2}t}), 
	\end{align*}
	which, by Plancherel's identity, implies for $t\ge 1$, $s \in \N_0$,
	$$\|I_2\|_s\le e^{-\theta_1t}\|f\|_s,$$
   and for $t \ge 1$, $\|I_2\|_s$ is clearly uniformly bounded. Thus, using Sobolev embedding and interpolation, we can conclude
   $$\|I_2\|_{L^p} \le Ce^{-\theta_1 t}\|f\|_{s_0},~~2 \le p \le \infty.$$
\end{proof}

Lastly, by Duhamel's formula and standard calculations, we obtain the following corollary.

\begin{coro}
	\label{decacy:lin}
	For all $T\in [0,\infty)$, $U_0 \in (H^3 \times H^2) \cap (W^{1,1} \times L^1)$, $f_1 \in L^2([0,T],H^2 \cap L^1)$, $f_2 \in C([0,T],H^3 \cap L^1)$, $f_3 \in W^{1,1}([0,T], L^2) \cap L^2([0,T], H^2 \cap L^1)$ the solution  $U$ of
	$$
		U_t-LU=F_1+\partial_{x} F_2+\partial_t F_3,~~U(0)=U_0,~~F_j=(0_{\R^N},f_j)^t,~j=1,2,3,
	$$
	satisfies for all $t \in [0,T]$, $\theta>0$ and some $c>0$,
	\begin{equation}
		\label{decaylin}
		\begin{split}
			\int_0^t e^{-\theta(t-s)}\|U(s)\|^2_{1,0} ds & \le C(1+t)^{-\frac{d-1}2}(\|U_0\|_{3,2}^2+\|U_0\|_{W^{1,1} \times L^1}^2+\|f_3(0)\|^2_{H^2 \cap L^1})\\
			& \quad +C\int_0^te^{-c(t-s)}(\|f_1(s)\|_2^2+\|\partial_x f_2(s)\|_{2}^2+\|f_3\|_2^2)ds\\
			& \quad+C \left(\int_0^t(1+t-s)^{-\frac{d-1}4}(\|f_1(s)\|_{L^1}ds\right)^2\\
			&\quad +C\left(\int_0^t(1+t-s)^{-\frac{d-1}{4}-\frac12}(\|f_2(s)\|_{L^1}+\|f_3(s)\|_{L^1})ds\right)^2.
		\end{split}
	\end{equation}
\end{coro}

\section{Resolvent estimates}\label{s:4}

The purpose of this section is to prove Proposition \ref{prop:resolvent} in order to complete the linear analysis.

\subsection{Small frequencies.} We start with a simple but important remark.
\begin{remark}
	\label{rem:hatU}
	Let $\hat U=(\hat U^1,\hat U^2)$ satisfy \eqref{flt2} for $\hat F=(\hat F^1,\hat F^2)$. Then $\hat U^1$ satisfies  \eqref{flt} with right hand side
	$$\hat f=(L^{22}_\eta-\lambda I)\mathcal A \hat F^1-\hat F^2,$$
	and $\hat v=\mathcal A^{-1}\hat U_2$ satisfies \eqref{flt} with
	$$\hat f=L^{21}_\eta \hat F^1 +\lambda \hat F^2.$$
\end{remark}

Thus, for bounded $|(\lambda,\eta)|$, estimates on \eqref{flt} directly imply estimates for \eqref{flt2}, which leads us to investigate the former equation or more precisely the equivalent first-order ODE \eqref{lin2}.  To this end, we follow the argumentation in \cite{GMWZ05}. First, write \eqref{lin2} as a \textit{double boundary value problem}. For a function $h$ defined on the real axis, set
$$h_+(x)=h(x),~~h_-(x)=h(-x), ~~x \ge 0,$$
and
$$\mathcal V=\begin{pmatrix} V_+\\ V_-\end{pmatrix},~~\mathcal G=\begin{pmatrix}
	 G_+ & 0\\
	0 &-  G_-
\end{pmatrix},~~\mathcal F=\begin{pmatrix}  \hat F_+\\- \hat F_-\end{pmatrix},~~ \Gamma  \mathcal V= V_+- V_-.$$
Then \eqref{lin2} is equivalent to
\begin{equation}
	\label{bvp}
	\begin{aligned}
	&\mathcal V_x-\mathcal G\mathcal V=\mathcal F,&\text{for}~ x \ge 0,\\
	&\Gamma \mathcal V=0,&\text{for}~ x=0.
	\end{aligned}
\end{equation}

The key part in obtaining estimates for solutions of \eqref{bvp} is to study the coefficient matrix of \eqref{bvp} at $x= \infty$, namely
$$\mathcal G(\infty,\lambda,\eta)=\begin{pmatrix}  G(u_+,\lambda,\eta) & 0\\
	0 & - G(-u_-,\lambda,\eta)\end{pmatrix}.$$

For notational purposes, we introduce the Fourier-Laplace symbol
\begin{align*}B'(u,\zeta)&=-\lambda^2\mathcal A(u)+i\lambda C_2(u,\eta)-B_{22}(u,\eta).\\
\end{align*}

Directly from Proposition \ref{prop:spec1}, we obtain the following corollary.

\begin{coro}
	\label{coro:spec2}
	\begin{enumerate}
		\item[(i)] There exists $c>0$ such that for $\zeta \in M_c$  the matrix $\mathcal G(\infty,\zeta)$ has (counting multiplicities) $2n$ eigenvalues with positive and $2n$ eigenvalues with negative real part.
		\item[(ii)] $0$ is a semi-simple eigenvalue of multiplicity $2n$ of $\mathcal G_\infty(0)$. The other eigenvalues (fast modes) are those of $A^1_+(B^{11}_+)^{-1}$ and $-A^1_-(B^{11}_-)^{-1}$. 
		\item[(iii)] For all $c_1>0$, $\hat \gamma\ge c_1$, any eigenvalue $\mu$ of $\mathcal G(\infty,0,\hat \zeta)$ splitting from $0$ (slow mode) admits the expansion
		\begin{equation} 
			\label{mu1}
		\mu(\rho,\hat \zeta)=\rho\mu_1(\hat \zeta)+O(\rho),~~\rho \to 0,
		\end{equation}
		where $\mu_1$ is an eigenvalue of 
		\begin{equation} 
			\label{mu2}
		\mp(A^1_\pm)^{-1}(\hat \lambda A^0_\pm+i A_{2\pm}(\hat \eta)).
		\end{equation}
		In particular, $\Rep \mu_1 \neq 0$.
	\end{enumerate}
\end{coro}

Corollary \ref{coro:spec2} allows us to reduce the investigation of \eqref{bvp} to that of the limiting system, by using a so-called MZ-conjugator, whose existence follows from Corollary \ref{coro:spec2}  by virtue of the Gap Lemma 
\cite{GZ98,Z01,MZ05}.

\begin{lemma}
	\label{lem:con}
	For any $C>0$, $\Omega=S^d_c \cap \{\rho \le C\}$, there exists a smooth uniformly bounded matrix family $K \in \C^\infty([0,\infty) \times \Omega); \operatorname{Gl}_{4n})$ such that the family $K(x,\rho,\hat \zeta)^{-1}$ is also uniformly bounded, and
	\begin{enumerate}
		\item[(i)] for some $\theta>0$, $K(x,\rho,\hat\zeta)=I+O(e^{-\theta x})$,
		\item[(ii)] $K$ satisfies
		$$K_x(x)=\mG(x)K(x)-K(x)\mG(\infty).$$
	\end{enumerate}
\end{lemma}
Clearly, $\mathcal V$ satisfies \eqref{bvp} if $\tilde \mW=K^{-1}\mathcal V$ satisfies
\begin{equation} 
	\label{bvp2}
	\begin{aligned}
		\tilde \mW_t-\mG(\infty,\zeta)\tilde \mW=K^{-1}\mF,& ~~x>0\\
		\Gamma_1(x,\zeta)\tilde \mW=0,&~~x=0,
	\end{aligned}
\end{equation}
with $\Gamma_1(x,\zeta)=\Gamma K(x,\zeta).$.

In the following, we fix a base-point $\underline{\hat \zeta} \in S^d_+$ and consider an open $\overline{S^d_c}$-neighbourhood $\Omega_0$ of $(0,\underline{\hat \zeta})$, which we will shrink if necessary. For any matrix $M \in \C^{n \times n}$, we write 
$$\check M=(A^0)^{-\frac12}M(A^0)^{-\frac12}.$$

We obtain the following decomposition of $\mathcal G(\infty)$. 

\begin{lemma}
	\label{lem:decomp}
	 There exists smooth family of basis transformation $T_1 \in C^\infty(\Omega_0,\operatorname{Gl}_{4n}(\C))$ such that for all $(\rho,\hat \zeta) \in \Omega_0$,
	$$\mathcal G_1(\rho,\hat \zeta):=T_1(\rho,\hat \zeta)^{-1}\mathcal G(\infty,\rho,\hat\zeta) T_1(\rho, \hat\zeta)=\begin{pmatrix} P_+(\rho,\hat \zeta)& 0& 0\\
	0&	P_-(\rho,\hat\zeta) & 0\\
	0& 0 &  H(\zeta)\end{pmatrix},$$
   where $P_\pm \in C^{\infty}(\Omega, \C^{k_{\pm} \times k_{\pm}})$, $k_++k_-=2n$, $H \in C^\infty(\Omega,\C^{2n \times 2n})$, such that
    \begin{align*}&\Rep P_+ \ge cI,\quad \Rep P_- \le cI\\
    &H(\rho,\hat \zeta)=\rho\tilde H(\hat \zeta,\rho)=\begin{pmatrix}
    	\rho H_+(\rho,\hat \zeta) &0\\
    0 &	\rho H_-(\rho,\hat \zeta)
    \end{pmatrix},\quad H_{\pm}(\rho,\hat \zeta)= H_{0\pm}(\hat \zeta)+\rho H_{1\pm}(\hat \zeta)+O(\rho^2),\\
&H_{0\pm}(\hat \zeta)=\mp(\check A^1_\pm)^{-1}(\hat\lambda I +i\check A_{2\pm}(\hat \eta)),\\
&H_{1\pm}(\hat \zeta)=\left(H_{0\pm}(\hat \zeta)^2(\check A_\pm^1)^{-1}\check B_\pm^{11}+H_{0\pm}(\hat \zeta)(\check A_\pm^{1})^{-1}(i\check B_{\pm12}(\hat \eta)+\hat \lambda \check C^1 +(\check A_\pm^1)^{-1}\check B'_\pm(\hat \zeta)\right)
	\end{align*}
\end{lemma}

\begin{proof}
	This follows directly from block diagonalizing $G_+$ and $G_-$ with respect to the slow and fast modes, applying matrix perturbation  
	theory \cite{Kat85} using \eqref{mu0}, \eqref{mu1}, \eqref{mu2} and performing the basis transformation $K \to (A^{0})^{\frac12}K(A^0)^{-\frac12}$.

\end{proof}

The key feature of the expansion $H_\pm=\rho H_{0\pm}+\rho^2H_{1\pm}$ is the following.

\begin{lemma}
	\label{lem:h0} Let $(\tau,\eta) \in \R^d\setminus\{0\}$. Then
	 $\mu=i\xi_1=i\xi_1(\tau,\eta)$ is a purely imaginary eigenvalue of $H_{0\pm}(0,\tau,\eta)$  if and only if $\tau$ is an eigenvalue of $-\check A(\xi)$, $\xi=(\pm \xi_1, \eta)$. In this case, $(\xi_1,\eta) \neq 0$, and there exists  a right projector $R_\tau$ on the eigenspace corresponding to $\tau$ with left-projector $L_\tau$  such that
	$$\pm\Rep (L_\tau\check A^1H_{1\pm}(\tau,0,\eta) R_\tau) < -c|(\xi_1(\tau,\eta),\eta)|.$$
\end{lemma}

\begin{proof}
	We treat only the '+`-case and suppress the subscript.
	
	For an eigenvalue $\mu=i\xi_1$, $\xi_1 \in \R$, of $H_0(\tau,0,\eta)$ with eigenvector $v$,  it holds
	$$\big(i(\check A^1)^{-1}(\tau+\check A_2(\eta))+i\xi_1\big)v=0,$$
	which is clearly equivalent to
	\begin{equation} 
		\label{xi1}
		 i(\tau I +\check A_{2}(\eta)+\xi_1\check A^1) v=0.
	\end{equation}
	Naturally, $(\xi_1,\eta) \neq 0$ as otherwise $(\tau,\eta)$ would vanish as well, which contradicts the assumption. 
	
	Now, let $R_{\tilde \tau} =R_{\tilde \tau}(\tilde \xi)=R_\tau (\xi)$, $(\tilde \tau,\tilde \xi_1,\tilde \eta)=|\xi|^{-1}(\tau,\xi_1,\eta)$, be a right-projector on the eigenspace of $\tau$ with corresponding left-projector $L_{\tilde \tau}$ such that 
   \begin{equation} 
   	\label{rep100}
   \operatorname{Re}\big(L_{\tilde \tau}(-\check B(\tilde \xi)+ \check{\mathcal A}\tilde \tau^2-\tilde \tau \check C(\tilde \xi)) R_{\tilde \tau}\big) < -cI,
   \end{equation}
   which exists due to condition (D1) (cf. also Remark \ref{rem:d1}).
   By definition, we have
\begin{equation}
	\label{xi100}
  |\xi|^2\left(-\check B(\tilde \xi)+\tilde \tau^2 \check{\mathcal A}-\tilde \tau\check C(\tilde \xi)\right)=\check B'(i\tau,\eta)- \xi_1^2\check B^{11}-\xi_1(\check B_{12}(\eta)+\tau C^1),
 \end{equation}
  and
  \begin{equation} 
  	\label{ltau}
  	L_\tau \check A^1\xi_1=-L_\tau(\tau+A_2(\eta))=-iL_\tau \check A^1 H_0.
  \end{equation}
  Now, \eqref{xi100}, \eqref{ltau} imply
  \begin{equation}
  	\label{h1}
  	L_{\tilde \tau}(-\check B(\tilde \xi)+ \check{\mathcal A}\tilde \tau^2-\tilde \tau \check C(\tilde \xi)) R_{\tilde \tau})=L_\tau\check A^1H_{1}(\tau,0,\eta) R_\tau,
  \end{equation}
	 and multiplying \eqref{rep100} by $|\xi|^2$ yields the result.
\end{proof}

We further analyse the structure of $\tilde H$. First, we block diagonalize  $\tilde H$ via smooth basis-transformation $ T_{H1}$ to obtain
\begin{equation} 
	\label{q}
	T_{H1}^{-1}\tilde H T_{H1}=\begin{pmatrix}  Q_1 & \cdots & 0\\ \vdots & \ddots & \vdots\\
		0 & \ldots &	Q_{\underline k}\end{pmatrix},
\end{equation}
where each block $Q_k \in \C^{\nu_k \times \nu_k}$ corresponds to the group of eigenvalues splitting from some eigenvalue $\underline \mu_k$ of 
$ H_+(\underline{\hat \zeta})$ or $H_-(\underline{\hat \zeta})$. We order these eigenvalues in such a way that $\mu_1,\ldots,\mu_{\underline k_1}$ have positive, $\mu_{\underline k_1+1},\ldots \mu_{\underline k_2}$ have negative and $\mu_{\underline k_2+1}, \ldots \mu_{\underline k}$ have $0$ real part.  The key part in the analysis is to characterize the behaviour of the latter. If $\underline \mu_k=i \underline \xi_1$, $ \underline \xi_1 \in \R$, is a purely imaginary eigenvalue of $H_{0\pm}$, then $\underline \gamma=0$  by Lemma \ref{lem:h0}, and
$$p(\underline{\hat \tau},\pm \underline{\xi_1},\underline{\hat \eta})=0.$$
 We call the block $Q_k$ corresponding to this eigenvalue  glancing if for some $1 \le r \le \bar r$, $\pm\underline \xi_1=\xi_{1r}(\underline{\hat \eta})$, where the $\xi_{1r}$ are defined below condition (S5). In particular, $(\underline{\hat \tau}, \pm\underline{\hat\eta})$ is in the glancing set with $\pm\underline{\hat \tau}=a_{lr}^\pm (\underline{\hat \eta})$, $a_{lr}^\pm$ defined ibid.. We call the block non-glancing otherwise.  In conclusion, by  appropriately choosing $T_{H1}$, we may assume the following decomposition
 \begin{equation} 
 	\label{q2}
 	T_{H1}^{-1}\tilde H T_{H1}=\begin{pmatrix} Q_{e+} & 0& 0 & 0\\ 0 & Q_{e-} & 0 & 0\\
 		0 & 0 &	Q_{h} & 0\\
 	0&0&0& Q_g\end{pmatrix},
 \end{equation}
 where
 $$\Rep Q_{e+}>CI, \quad \Rep Q_{e-} <CI,$$
sometimes also called elliptic modes, $Q_h$ consists of the non-glancing blocks corresponding to purely imaginary eigenvalues, hyperbolic modes, and $Q_g$ consists of the glancing blocks $Q_{g1},\ldots,Q_{g\bar r}$. For a glancing block $Q_{gr}$ corresponding to an eigenvalue $\pm i\underline\xi_{1r}$, we denote by $\bar s_r$, the multiplicity of the root $\underline\xi_{1r}$ with respect to 
$$p(\underline{\hat \tau},\pm \cdot,\underline{\hat \eta})=0$$ and set
\begin{equation}
	\label{alpha}
\nu_{lr}=(|\hat \tau-a_{lr}^\pm(\hat \eta)|+\hat \gamma +\rho)^{1/\bar s_r},~\beta=\max_{lr} \nu_{lr}^{1-\bar s_r},~\alpha=\max_{lr} \nu_{lr}^{1-[(\bar s_r+1)/2]}.
\end{equation}
 We denote by $Q_{gr-}$ (resp. $Q_{gr+}$) the space corresponding to eigenvalues with negative (resp. positive) real part for $\rho>0$ and set
$$G_{g\pm}=\bigoplus_{r=1}^{\bar r}  Q_{gr\pm}.$$

Our main result is the following (cf. the corresponding result in \cite{GMWZ05} for 'parabolic viscosity`).

\begin{prop}
	\label{prop:dec}
	\begin{enumerate}
		\item[(i)] There exists a smooth basis transformation $T_h$ on $\Omega_0$, such that $T_h, T_h^{-1}$ are uniformly bounded and
		$$T_hQ_hT_h^{-1}=\begin{pmatrix}
			Q_{h+} & 0\\
			& Q_{h-}
		\end{pmatrix}$$
		 with
		$$\Rep Q_{h+} \ge C(\rho+\hat \gamma),\quad  \Rep Q_{h-} \le -C(\rho+\hat \gamma)$$
		\item[(ii)] 
		There exists a basis transformation $T_g$ on  $\Omega_0 \setminus\{\rho=\hat \gamma=0\}$ with
		$$|T_g| \le C,\quad |(T_g)^{-1}| \le C\beta\quad |(T_g|_{Q_{g-}})^{-1}| \le C\alpha,$$
		\begin{equation} 
			\label{tg1}
			T_gQ_gT_g^{-1}=\begin{pmatrix}
			Q_{g+} & 0\\
			& Q_{g-}
		\end{pmatrix},
	\end{equation}
		with
		\begin{equation} 
			\label{tg2}\Rep Q_{g+} \ge C(\rho+\hat \gamma)\beta,\quad  \Rep Q_{g-} \le -C(\rho+\hat \gamma)\beta.
				\end{equation}
	\end{enumerate}
\end{prop}

\begin{proof}
	We follow \cite[proof of Lemma 4.8]{Z01} and \cite[proof of Lemma 12.1]{GMWZ05}.
	 It is clearly sufficient to treat a single block $Q_{gr}$. W.l.o.g. assume that $Q_{gr}$ arises from $H_{0+}$ and suppress the subscript $gr$ in the following.
	
	 Let $i\underline \xi_{1}$ be a purely imaginary eigenvalue of $H_0(\underline{\hat \zeta})$. Then $\underline{\hat \gamma}=0$, and $\underline{\hat \tau}$ is an eigenvalue of $-\check A(\underline \xi)$, $\underline \xi=(\underline \xi_{1}, \underline{\hat \eta})$. In particular, $\underline{\hat \tau}= a_l(\underline \xi_1,\underline{\hat \eta})$ for some $1 \le l\le \bar l$. In the following, we just write $a$ for $a_l$. For $\xi=(\xi_1, \hat \eta)$ close to $(\underline \xi_1,\underline{\hat \eta})$, we choose a smooth basis  $r_1(\xi),\ldots,r_{m}(\xi)$ of the kernel of $a(\xi)I +\check A(\xi)$ such that the projector $R(\xi)=(r_1(\xi),\ldots,r_m(\xi))$ satisfies the properties of Remark \ref{rem:d1}. $L(\xi)=(l_1(\xi ),\ldots,l_m(\xi))$ is the projector on the left-kernel of $a(\xi)I+\check A(\xi)$ dual to $R(\xi)$.

	(ii) By (S5), for $\hat \tau,\hat \eta$ close to $\underline{\hat \tau}, \underline{\hat\eta}$, $\xi_{1r}(\hat \eta)$ is the unique root of $\hat \tau -a(\cdot,\hat \eta)$ close to $\underline \xi_1=\xi_{1r}(\underline{\hat \eta})$. 
	Additionally, the multiplicity $\bar s=\bar s_r$ of $\xi_{1r}(\eta)$ is constant, i.e.
	$$
	\partial_{\xi_1}^s a(\xi_{1r}(\hat \eta),\hat \eta)=0 \text{ for } s=1,\ldots,\bar s-1, \quad \partial_{\xi_1}^{\bar s}a( \xi_{1r}(\hat \eta), \eta) \neq 0.
	$$
	The inviscid theory (\cite{M01}, \cite{Z01}) tells us that  $i \xi_{1r}(\hat \eta)$ splits for $|\hat \tau-a(\hat\eta)|+\hat \gamma>0$ (non-smoothly) into  $m$ copies of $\bar s$ roots $\alpha_1(\hat \tau,\hat \eta),\ldots,\alpha_{\bar s}(\hat \tau,\hat \eta)$ of 
	$$
	f_{\hat \tau,\hat \eta}(\alpha)=\hat \tau+a(\alpha,\hat \eta).
	$$ 
	In addition, for
	$$
		r_{j,s}(\hat \eta)=(-1)^s\partial^s_{\xi_1}r_j(\tilde \xi_{1r}(\hat \eta),\hat \eta),~~j=1,\ldots,m~s=0,\ldots,\bar s-1,
	$$
	the family
	\begin{equation} 
		\label{jb}
		\{r_{1,0},\ldots,r_{1,\bar s-1},\ldots,r_{m,0},\ldots,r_{m,\bar s-1}\}
	\end{equation}
	is a Jordan basis of the generalized eigenspace 
	$$
		-(A^1)^{-1}(a(\hat \eta)+\check A_2(\hat \eta))
	$$
	corresponding to $\xi_{1r}(\hat \eta)$, and the genuine left eigenvectors of the dual basis are given by 
	\begin{equation} 
		\label{left}
		\tilde l_j=1/\bar s!(\partial_{\xi_1}^{\bar s}a)^{-1}l_j\check A^1=:p^{-1}l_jA^1,
	\end{equation}
	where $L=(l_1(\hat \eta),\ldots,l_m(\hat \eta))$ is the dual basis to $(r_1(\hat \eta),\ldots,r_m(\hat \eta))$ of the left-kernel of $a(\xi_{1r}(\eta),\eta)+A(\xi_{1r}(\hat \eta),\hat \eta).$
	We define the right projector corresponding to the Jordan basis  \eqref{jb} as
	$$
		\tilde R_j=(r_{j,0},\ldots,r_{j,\bar s-1}),~j=1,\ldots, m,\quad \mathcal R:=(\tilde R_1,\ldots, \tilde R_m),
	$$
	and denote by 
	$$
		\mathcal L=\begin{pmatrix}\tilde L_1 \\ \vdots \\ \tilde L_m \end{pmatrix} , \tilde L_j \in \R^{\bar s \times n}
	$$ 
	the dual left-projector. We find
	\begin{equation} 
		\label{qj1}
		\mathcal L(\hat \eta)\tilde H_0(a(\hat \eta),\hat \eta) \mathcal R(\hat \eta)=
		\begin{pmatrix}
			q(\hat \eta) & \cdots & 0 \\
			\vdots & \cdots & 0\\
			0 & \cdots & q(\hat \eta)
		\end{pmatrix},
	\end{equation}
	where 
	\begin{equation}
		\label{qj}
		q=i(\xi_{1r}(\hat \eta)+J),~~J=
		\begin{pmatrix}
			0& 1 & 0 & \cdots \\
			\vdots & \cdots & \ddots  & \vdots \\
			0 &  \cdots &   0 & 1\\
			0 & \cdots & \cdots & \cdots 
		\end{pmatrix} \in \R^{\bar s \times \bar s}.
	\end{equation}
   Thus, with $p:=\partial_{\xi_1}^{\bar s} a(\xi_{1r}(\hat \eta),\hat \eta)$, $\sigma=(\hat \tau-a(\hat \eta))$,
   \begin{align*}
   		\mathcal L(\hat \eta)\tilde H(\rho, \hat \gamma,\hat \tau,\hat \eta) \mathcal R(\hat \eta)&=
   		\begin{pmatrix}
   			q(\hat \eta) & \cdots & 0 \\
   			\vdots & \cdots & 0\\
   			0 & \cdots & q(\hat \eta)
   		\end{pmatrix}+
   		\begin{pmatrix}
   			M^{11} & \cdots & M^{1m}\\
   			\vdots & \cdots & \vdots\\
   			M^{m1} &  \cdots & M^{ m m}
   		\end{pmatrix}\\
   		&\quad+O(\rho^2+\hat \gamma^2+|\sigma|^2),
	\end{align*}
	where 
	$$
		M^{jk}=\tilde L_j\big(-(i\sigma+\hat \gamma)(\check A^1)^{-1}\mathcal + \rho H_1(a(\xi_{1r}(\hat \eta),0,\hat \eta) \big)\tilde R_k \in \C^{\bar s\times \bar s}.
	$$
	If we denote by $\theta_{jk}$ the lower left corner of $M^{jk}$, and set
	$$
		\Theta=
		\begin{pmatrix}
	 		\theta_{11} & \cdots & \theta_{1m}\\
	 		\vdots & \cdots & \vdots\\
	 		\theta_{m1} & \cdots & \theta_{m m}
		\end{pmatrix},
	$$ 
	we obtain by standard matrix perturbation theory that the eigenvalue of $\mathcal L \tilde H \mathcal R$ splits as
	$$
		\alpha_{s}^j=i\xi_{1r}(\hat \eta)+\pi_s^j(\rho,\sigma,\hat \eta)+o(|\rho|+|\hat \gamma|+|\sigma|),~0\le s \le \bar s-1, 1 \le j \le m,
	$$
	where
	$$
		\pi_s^j=(1)^{s/\bar s}i(b_j)^{1/\bar s},
	$$
	$b_1,\ldots,b_{m}$ being the eigenvalues of $-i\Theta$.
	We compute
	\begin{align*} 
		\theta_{jk}&=\tilde l_j \big(-(i\sigma+\hat \gamma)(\check A^1)^{-1}+ H_1(0,a(\hat \eta),\hat \eta)\big)r_k\\
		&=-p^{-1}\big((i\sigma+\hat \gamma)l_{j}^tr_k- l_j\rho\check A^1H_{1}(0,a(\hat \eta),\hat \eta)r_k\big),
	\end{align*}
	which means
	$$
		\Theta=-p^{-1}\big( (i\sigma+\hat \gamma)I_{\bar s}-\rho LA^1H_{1}(0,a(\hat \eta),\hat \eta)R\big),
	$$
	and
	$$
		b_j=p^{-1}(\sigma-i(\hat \gamma-\rho \mu_j(\hat \eta))),
	$$
	i.e.,
	\begin{equation} 
		\label{pi}
		\pi_s^j=(1)^{s/\bar s}i\big(p^{-1}(\sigma(\hat\eta)-i(\hat \gamma-\rho \mu_j(\hat \eta)))\big)^{1/\bar s},
	\end{equation}
	where $\mu_j$ is an eigenvalue of $Q=L\check A^1H_{1}(a(\hat \eta),0,\hat \eta)R$.  
	By Lemma \ref{lem:h0} and due to continuity of $\xi_{1r}(\hat \eta)$,
	$$
		\Rep \mu_j<-c_1(|\xi_{1r}(\hat \eta),\hat \eta|^2) \le -c_2.
	$$  
	 On $S^d_c$, $c>0$ sufficiently small, we have
	$$
		\hat \gamma-\Rep \mu_j\rho > \hat \gamma +c_2\rho \ge \hat \gamma+c c_2\rho(|\hat \eta|^2+|\hat \tau|^2) \ge 0.
	$$
	 For $z=z_1+iz_2$, $z_1 \in \R, z_2>0$, standard estimates yield 
	 $$
	 |\operatorname{Im} (z_1-iz_2)^{1/s}| =O(z_2|z|^{1/s-1}), \quad |\Rep (z_1-iz_2)^{1/s}|=\begin{cases}
	 	O(z_2|z|^{1/s-1}), & s=2,\\
	 	O(|z|^{1/s}), &s \ge 3,
	 \end{cases}
	 $$
	and we obtain
	$$
		|\Rep \pi_s^j| \ge C(\rho+\hat \gamma)\beta,
	$$
	giving \eqref{tg2} for some smooth (but not bounded) transformation. 

	For the last part of the proof, let $1 \le j \le m$. We note that by standard calculations, it follows from \eqref{pi} that the index set $S=\{s_1,\ldots,s_{\bar l}\}$ of $\pi^j_{s_l}$ with strictly negative real part is given as
	$$
		\begin{cases}
			\{1,\ldots,[\bar s/2]\} & p>0\\
			\{0,\ldots,[(\bar s-1)/2]\} & p<0\\
		\end{cases},
	$$
	in particular $s_-=|S| \le [(\bar s+1)/2]$. 

	Due to condition (S6), by choosing appropriate eigenvectors, we can assume 
	$$
		LA^1H_{1}(0,a(\hat \eta),\hat \eta)R=\operatorname{diag}(\mu_1,\ldots,\mu_{m}).
	$$
	Then a basis of eigenvectors of $\mathcal LH\mathcal R$ is given by
	$$v_s^j=\big(0_{(j-1)\bar s},1,\pi_{s}^j,(\pi_{s}^j)^2,\ldots,(\pi_s^j)^{\bar s-1},0_{\bar s(m-j+1)}\big)^t+o((|\sigma|+\rho)^{1/s}).$$
	Thus, the existence of a basis transformation $T_g$ with \eqref{tg1}, \eqref{tg2} and $|T_g| \le C$  follows from $\eqref{pi}$ and the boundedness of $r_1,\ldots,r_m$, $\partial_{\xi_1}^sa$. Clearly, $T_g$  is in leading order a block diagonal matrix
	$$
		T_g=\operatorname{diag}(T^1,\ldots,T^{m}),
	$$
 	with Vandermonde matrix blocks $T^j \in \C^{\bar s \times \bar s}$, whose generators are given by $\pi_0^m,\ldots,\pi_{\bar s-1}^m$. In particular,
	\begin{equation}
		\label{gen}
		|\pi_{s_1}^m-\pi_{s_2}^m| \ge c_1(\rho+\hat\gamma+|\sigma|)^{1/\bar s}=c_1\nu_k, 1 \le s_1 < s_1 \le \bar s
		\end{equation}
	and the estimate
	$$|(T_g)^{-1}| \le C\nu_{lk}^{1-\bar s}$$
	follows from the standard theorem about inverses of Vandermonde matrices. 

	Lastly, for a block $T^j$, $1 \le j \le m$, denote by 
	$$
		t=\begin{pmatrix}t_1 \\t_2 \end{pmatrix},
		~~t_1 \in \C^{s_- \times s_-}, t_2 \in \C^{\bar s-s_- \times s_-},$$
	the matrix whose columns are the eigenvectors in the block $T^j$ corresponding to stable eigenvalues. Clearly, $t_1$ is  a Vandermonde matrix of size $s_-\le [(\bar s+1)/2]$, whose generators also satisfy \eqref{gen}. Thus, 
	$$
		|t^{-1}| \le  c\nu_{lr}^{1-([(\bar s+1)/2]},
	$$ 
	and standard estimates (cf. e.g. \cite[proof of Lemma 12.1]{GMWZ05}) show 
	$$
		|(T^m|_{\mathcal H_{k-}})^{-1}| \le |t^{-1}| \le c\nu_{lr}^{1-([(\bar s+1)/2]}.
	$$
	Lastly, note that (i) is a special case of (ii) with $\bar s_r=1$, ``empty'' Jordan-blocks, and thus non-singular left and right eigenvectors.
\end{proof}

The reasoning below closely follows \cite[Section 12]{GMWZ05} or, more precisely, the simplified arguments \cite[Section 4]{N09}.

First, block-diagonalize the coefficient matrix $\mG$ via the  the basis transformations
\begin{align*}T_{H2}&=T_{H1} \begin{pmatrix} I_{\operatorname{dim}Q_{e+} +\operatorname{dim}Q_{e-}}& 0 & \\ 0 &T_h & 0 \\0&0&T_g\end{pmatrix},\quad \mT=T_1\begin{pmatrix}
	I_{2n}& 0 \\ 0 &T_{H2} 
\end{pmatrix}.
\end{align*}
Then by definition,
\begin{equation} 
	\label{g2}
	\mG_2:=\mT^{-1}\mG(\infty)\mT=\operatorname{diag}(P+,P-,Q_{e+},Q_
{e-},Q_{h+},Q_{h-},Q_{g+},Q_{g-}),
\end{equation}
and we get a corresponding decomposition of  $\C^{4n}$ into $\mG_2$-invariant subspaces
$$
	\C^{4n}=E_{P+}+E_{P-}+E_{e+}+E_{e-}+E_{h+}+E_{h-}+E_{g+}+E_{g-}.
$$
For notational purposes, define the index-set $J:=\{P,e,h,g\}$. Then, set
$E_{j}=E_{j+}+E_{j-}$, $j \in J$ and 
$$E_{\pm}=E_{P\pm}+E_{e\pm}+E_{h\pm}+E_{g\pm},\quad E_{H\pm}=E_{e\pm}+E_{h\pm}+E_{g\pm}.$$
Next, near our base point  $(0,\underline{\hat \zeta})$, set $\tilde{\mT}:=K\mT $,  where $K$ is the MZ-conjugator as constructed Lemma \ref{lem:con}. Now, for a solution $\mathcal V$ of \eqref{bvp}, the function $\mW=\tilde{\mT}^{-1}V$ satisfies
\begin{align} 
	\label{bvp3}
	\mW_t-\mG_\infty \mW=\bar{\mF},& ~~x>0,\\
	\label{bvp3bv}
	\Gamma_2(x,\zeta)\mW=0,&~~x=0,
\end{align}
with $\bar{\mF}=\tilde{\mT}^{-1} \mF$, $\Gamma_2=\Gamma\tilde{\mT}$.

For a vector $X \in \C^{4n}$, we use the decomposition
\begin{equation} 
	\label{v}
	X=\bigoplus_{j \in J}X_{j+}+X_{j-}, \quad X_{j\pm}\in E_{j\pm}.
\end{equation}
Due to the block-diagonal form of $\mG_2$, $\mW$ satisfies \eqref{bvp3} if and only if for all $j \in J$,
\begin{equation}
	\label{blockW}
	\mW_{j\pm}-Q_{j\pm}=\bar{\mF}_{j\pm},
\end{equation}
where $Q_{P\pm}=P\pm$.  We arrive at the following key estimate.

\begin{lemma}[\cite{N09}]
	\label{lem:estW}
	For $\bar{\mF} \in L^1$, each solution $\mW$ of \eqref{bvp3} with $W(\infty)=0$ satisfies
	\label{lem:decW}
\begin{align}
	\label{W+}
	|\mW_{j+}|^2_{L^\infty}+\kappa_{j}|\mW_{j+}|_{L^2}^2 &\le C |\bar{\mF}_{j+}|_{L^1}^2,\\
	\label{W-}
	|\mW_{j-}|^2_{L^\infty}+\kappa_{k}|\mW_{j-}|_{L^2}^2 &\le  C| \mW_{j_-}(0)|^2+ |\bar{\mF}_{j
		-}|_{L^1}^2\\
\end{align}
where $\kappa_j=1,\rho,\rho^2,\beta\rho^2$ for $j=P,e,h,g$. 
\end{lemma}

\begin{proof}
	We take the scalar product in $\C^{4n}$ of \eqref{blockW} with $\pm\mW_{j\pm}$  and integrate from  $x$ to $\infty$ in the ``$+$-case'' and from $0$ to $x$ in the ``$-$-case''. Then the result follows from the decomposition \eqref{g2}, Lemma \ref{lem:decomp},  and Proposition \ref{prop:dec}.
\end{proof}

As will be specified below, the desired resolvent estimate \eqref{sfre} (i.e. estimates on solutions of \eqref{bvp}) follow in a rather straightforward manner from Lemma \ref{lem:estW} provided we control the boundary terms $|\mW_{j-}(0)|$. It is here, where the Evans-function condition (S5) comes into play.  

First, note that 
$$\mE_-(\rho,\hat \zeta):=\tilde{\mT}_-E_-,$$
defined near the base point, is the space of boundary values corresponding to exponentially decaying solutions of \eqref{bvp}, with $\mF=0$, for $\rho>0$. More precisely,
$$\mE_{P-}=\tilde{\mT}E_{P-},~~(\mE_{H-}(\rho,\hat \zeta)=\tilde{\mT}E_{H-})$$
are the spaces corresponding to fast (slowly) decaying solutions. Since  $P_{-}(0,\hat \zeta)$ does not depend on $\hat \zeta$, $\mE_-$ varies smoothly on $\overline{S^d_c}$. 

By definition, $\mathcal V=(V_+,V_-) \in \mE_-$ if and only if $V$ defined by
$$V_+(x)=V(x),\quad V_{-}(x)=V(-x),~~x \ge 0,$$  
is an exponentially decaying solution to \eqref{lin2} with $\mF=0$. Thus, the spectral stability assumption (S5) implies that for $\rho>0$, the intersection $\ker \Gamma \cap \mathcal E_{-}$ is trivial, which corresponds to the so-called Kreiss condition for hyperbolic BVPs. However, for $\rho=0$, $(\bar v(x),0)$ with $\bar v=\bar{u}'$ is a decaying solution of \eqref{lin2}. Thus, for all $\hat \zeta \in S^d_+$,
$$\mathcal E_{P-}(0,\hat \zeta) \cap \ker \Gamma(0,\hat \zeta) \neq \{0\}.$$
In fact,  since by (S5) the Evans function vanishes at exactly first order for $\rho=0$, this intersection is exactly the span of $(\bar v(0),0,\bar v(0),0)$.  The observations above are made precise in the following lemma, whose proof goes exactly as \cite[proof of Proposition 7.1]{GMWZ05}. (To see this, note that in the notation of \cite{GMWZ05}, $\mathcal E_{-,\phi} \subset \mathcal E_{P-}$, and $\mathcal E_{H-} \subset \mathcal E_{-,\phi,\eps}^c$.)

\begin{lemma}
	\label{lem:ef}
	\begin{enumerate}
		\item[(i)]
		There exists $C>0$, $\delta>0$ such that for all $(\rho,\hat \zeta) \in S^d_c \cap \{0\le \rho \le \delta\}$,
		\begin{align}
			\label{degkreiss}
			|\Gamma  \mathcal V| \ge C\rho|\mathcal V|, &\quad \mathcal V \in \mathcal E_{P-}(\rho,\hat \zeta)\\
			\label{kreiss}
			|\Gamma \mathcal V| \ge C|\mathcal V|, & \quad \mathcal V \in  \mE_{H-}(\rho,\hat \zeta)
		\end{align}
		\item[(ii)]
		For all $0<\delta<R$, there exists $C_{\delta R}$ such that for all $(\rho,\hat \zeta) \in  S^d_c \cap \{\delta \le \rho \le R\}$ and  $\mathcal V \in \mathcal E_{-}(\rho,\hat \zeta)$,
		$$|\Gamma \mathcal V| \ge C_{\delta R}|\mathcal V|.$$
		\end{enumerate}
\end{lemma}

In variables $\mW=\tilde{\mT}^{-1}V$, this gives.

\begin{lemma}
	\label{lem:estk}
	There exists $\delta>0$ such that for all $(\rho,\hat \zeta) \in \Omega_0$, $\mW_- \in E_-(\rho,\hat \zeta)$,
	$$|\Gamma_2 \mW_-| \ge \delta(\rho|\mW_{P-}|+|\mW_{e-}|+|\mW_{h_-}|+\alpha^{-1}|\mW_{g-}|),$$
	with $\alpha$ defined in \eqref{alpha}.
\end{lemma}

\begin{proof}
	By definition, for $\mW \in \C^4$, we have $\mW \in E_{j\pm}$ if and only if $\mathcal V=\tilde{\mathcal T}\mW \in \mE_{j\pm}, j \in J$. Thus for $\mW \in E_{P-}$, we find by definition of the basis transformation $\tilde{\mT}$, the boundedness of $T_1^{-1}$ (Lemma \ref{lem:decomp}) and \eqref{degkreiss},
	$$|\Gamma_2 \mW|=|\Gamma \mathcal V| \ge C\rho |\mathcal V|=C\rho|K\operatorname{diag}(T_1,I_{2n})\mW| \ge C\rho|\mW|.$$
	Analogously, the boundedness of $T_{H1}^{-1}, T_{h}^{-1}$ (Proposition \ref{prop:dec}) and \eqref{degkreiss} imply for $\mW \in E_{e-}$ or $\mW \in E_{h-}$ 
	$$|\Gamma_2 \mW|=|\Gamma \mathcal V| \ge C |\mathcal V| \ge C|\mW|.$$
	Lastly, for $\mW \in E_{g-}$, we get from Proposition \ref{prop:dec}
	$$|\Gamma_2 \mW| \ge C |KT_1T_{H1}\operatorname{diag}(I,|(T_g|_{Q_{g-}})^{-1}|\mathcal W| \ge C\alpha^{-1}|\mW|.$$
\end{proof}

On the other hand, $\Gamma_2\mW|_{x=0}=0$ implies that for decaying solutions $\mW$ of \eqref{bvp3}, \eqref{bvp3bv}, we have
\begin{equation} 
	\label{Gamma2}
	|\Gamma_2 \mW_-| \le |\Gamma_2 \mW|+|\Gamma_2 \mW_+| \le C|\mW_+|_{L^\infty}.
\end{equation}
For each $j$, we multiply  \eqref{W+} by a sufficiently large constant $\delta^{-1}$,
  \eqref{W-} by $\rho^{2},1,1,\alpha^{-2}$ for $j=P,e,h,g$, and sum the resulting equations to find 
$$
	\delta^{-1}|\mW_+|^2_{L^{\infty}}-C_1(\rho^2|\mW_{P-}(0)|^2+|\mW_{e-}(0)|^2+|\mW_{h-}(0)|^2+\alpha^{-2}|\mW_{g-}(0)|^2)+\rho^2|\mW|_{L^2}^2 \le C|\bar{\mF}|_{L^1}^2,
$$
where we also used $\rho \le \alpha^{-1} \le 1$ and $\beta\alpha^{-2} \ge 1$. Thus, Lemma \ref{lem:estW} and \eqref{Gamma2} yield for $\delta$ sufficiently small
\begin{equation} 
	\label{decW}
\rho^2(|\mW|_{L^\infty}^2+|\mW|_{L^2}^2) \le C|\bar{\mF}|_{L^1}^2.
\end{equation}
In terms of solutions $\mathcal V$ to the original problem \eqref{bvp}, this gives the following crucial estimate.
\begin{prop}
	\label{prop:dec2}
	There exist $c,c_1>0$ such that for $(\rho,\hat \zeta) \in S^d_c$ with $\rho \le c_1$ and $\mF \in L^1$, the solution $\mathcal V$ to \eqref{bvp} satisfies for $p \in [2,\infty]$
	$$
		|\mathcal V|_{L^p}\le C\frac{\beta|\mF|_{L^1}}{\rho}.
	$$
\end{prop}

\begin{proof}
	Due to compactness of $\overline{S^d_c} \cap \{\rho \le c_1\}$, we can argue locally. By the previous arguments, for a solution $\mathcal V$ to \eqref{bvp}, $\mathcal W=\tilde{\mathcal T}^{-1}\mathcal V$ satisfies \eqref{decW} with $\bar{\mF}=\tilde{\mT}^{-1}\mF$. Then, Lemmas \ref{lem:con}, \ref{lem:decomp} and Proposition \ref{prop:dec} yield
	$$|\mathcal V|_{L^2} \le |T_g| |\mathcal W| \le C\rho^{-1}|\bar{\mF}|_{L^1}\le C\rho^{-1}|T_g^{-1}||\mF|_{[L^1} \le C\beta\rho^{-1}|\mF|_{L^1}.$$
	The same is true with the $L^2$-norm replaced by the $L^\infty$-norm and thus for all $L^p$-norms by interpolation.
\end{proof}
This implies the following estimates for a solution $\hat u$ to \eqref{flt}. 

\begin{coro}
	\label{coro:aux}
	There exist $c,c_1>0$ such that for $(\rho,\hat \zeta) \in S^d_c\cap\{\rho \le c_1\}$,  the solution
	$ \hat u$  to  \eqref{flt} satisfies:
	\begin{enumerate} 
	\item[(i)] For $\hat f \in L^1$ and all $p \in [2,\infty]$,
	$$|\hat u|_{W^{1,p}}\le C\rho^{-1}\beta|\hat f|_{L^1}.$$
	\item[(ii)] If $\hat f=g_x \in L^1$ for some differentiable function $g$,  then for $p \in [2,\infty]$,
	$$|\hat u|_{L^p} \le C\beta|g|_{L^1}.$$
	\end{enumerate}
\end{coro}

\begin{proof}
	Assertion (i) follows in a very straightforward manner. For a solution $\hat u$ to \eqref{flt}, $\hat V=(\hat u,B^{11}\hat u_x)$ solves \eqref{lin2} with $\hat F=(0,\hat f)$, and $\mathcal V=(V_+,V_-)$ solves \eqref{bvp} for $\mF=(F_+,-F_-)$. Thus, Proposition \ref{prop:dec2} gives
	\begin{equation} 
		\label{V1}
		|\hat u|_{W^{1,p}} \le C|V|_{L^p} =C|\mathcal V|_{L^p} \le C\rho^{-1}\beta|\mF|_{L^1} =C\rho^{-1}\beta|\hat f|_{L^1}.
	\end{equation}
 (ii) To prove this assertion, we follow the idea of Kreiss \cite{KK98} (cf. also \cite{GMWZ05,N09}),  and consider the auxiliary problem 
 \begin{equation}
 	\label{aux}
 	B^{11}(\bar{u})\hat w_x-A^1(\bar{u})\hat w=\hat g.
 \end{equation}
 As the stable and unstable spaces of $B^{11}(\bar u)^{-1}A^1(\bar u)$ are separated due to condition (S3),  \cite[Lemma 12.6]{GMWZ05} shows that there exists a solution $\hat w$ to \eqref{aux}  satisfying
 $$|\hat w|_{L^p} \le C|g|_{L^1}$$
 for $1 \le p \le \infty$, and we can deduce directly from the equation that
 $$|\hat w_x|_{L^1} \le C(|\hat w|_{L^1}+|\hat g|_{L^1}) \le C|\hat g|_{L^1}.$$
 Now, let $\hat u$ be a solution of \eqref{flt} with $\hat f$ replaced by $\hat g_x$. Then $ \tilde u=\hat u-\hat w$
 satisfies
 \begin{equation} 
 	\label{flt3}
 	\begin{aligned} 
 		&(B^{11}\tilde u_{x})_x+(S(\lambda,\eta) \tilde u)_{x}-s(\lambda,\eta)\tilde u\\
 		&\quad=-((S(\eta,\lambda)+\tilde A^1)\hat w)_x+s(\lambda,\eta)\hat w=:\zeta(\eta,\lambda,w,w_x),
 	\end{aligned}
 \end{equation}
and (i)  gives
$$
	|\tilde u|_{W^{1,p}} \le C\beta\rho|\zeta|_{L^1}.
$$
By definition, for small bounded $\rho$,
$$
 |\zeta(\eta,\lambda,w,w_x)|=|\zeta(\rho,\hat \eta,\hat \lambda, w,w_x)| \le C\rho (|w|+|w_x|) \le C\rho|g|_{L^1},
 $$
 and we conclude
 $$|\hat u|_{L^p} \le |\hat w|_{L^p}+|\tilde u|_{L^p} \le C \beta |\hat g|_{L^1}.$$
\end{proof}

From this, we obtain the estimates for solutions to $(\lambda-L_\eta)\hat U=\hat F$, i.e. $\hat U=(\lambda-L_\eta)^{-1}\hat F$ claimed in Proposition \ref{prop:resolvent} (i) as follows.

\begin{proof}[Proof of Proposition \ref{prop:resolvent} (i)]
	Let $\hat U=(\hat U^1,\hat U^2)$ be a solution to \eqref{flt2}. As seen in Remark \ref{rem:hatU}, $\hat U^1$ satisfies \eqref{flt} for
	$$\hat f=(L_{\eta}^{22}-\lambda I)\mathcal A \hat F^1-\hat F^2.$$
	We compute
	$$L^{22}_{\eta}F^1=(C^1\mathcal A^{-1}\hat F^1)_x-\tilde A^0 \mathcal A^{-1}\hat F^1.$$
	Thus for bounded $(\lambda, \eta)$, we find by Corollary \ref{coro:aux} (i)
	\begin{equation}
		\label{res100}
	|\hat U^1|_{L^p} \le C\rho^{-1}\beta|\hat F^1|_{L^1} \le C\rho^{-1}\beta(|\hat F^1|_{W^{1,1}}+|\hat F^2|_{L^1}).
	\end{equation}
	On the other hand, $\hat U^2$ satisfies \eqref{flt} with
	$$\hat f=L^{21}_\eta F^1+\lambda F^2.$$
	By definition,
	$$L^{21}_\eta \hat F^1=(B^{11}\hat F^1_x)_x+((iB_{12}(\eta)-A^1)\hat F^1)_x-(B_{22}(\eta)- i\tilde A_2(\eta))\hat F^1=:\zeta^1_x+\zeta$$
	By Corollary \eqref{coro:aux} (i), (ii) and linearity of the equations, we get
	\begin{equation}
		\label{res200} 
		|\hat U^2|_{L^p} \le C\beta(|\zeta^1|_{L^1}+|\zeta|_{L^1}+|\hat F^2|_{L^1}) \le C\beta\rho^{-1}(|\hat F^{1}|_{W^{1,1}}+|\hat F^2|_{L^1}).
	\end{equation}
	Adding \eqref{res100} and \eqref{res200} gives \eqref{sfre}. Estimate \eqref{sfre2} follows analogously, with $\hat F$ replaced by $\hat F_x$ and application of Corollary \eqref{coro:aux} (ii).
\end{proof}

\subsection{Intermediate frequencies.} Fix $0<\delta<R$ and a base point $(\underline \rho,\underline{\hat \zeta})$ with
$$\delta \le \rho \le R, \quad |\underline{ \hat\gamma}| \ge -\frac12c\delta(|\underline{\hat \tau}|^2+|\underline{\hat \eta}|^2),$$
for $c>0$ sufficiently small. Note that this in particular implies $(\rho,\underline{\hat \zeta}) \in S^{d}_c$. In this case, the result is more or less a much easier version of the one for low frequencies.
First, Corollary \ref{coro:spec2} directly yields a basis transformation $\mT \in C^\infty(\overline{\Omega_0},\C^{4n \times 4n})$, $\Omega_0$ being a $S^d_c$-neighbourhood of $(\underline \rho,\underline{\hat \zeta})$, such that
$$\mT^{-1}G(\infty)\mT \le \begin{pmatrix}
	P_+ & 0\\ 0& P_-
\end{pmatrix}~~\text{on}~~\Omega_0,$$
for smooth matrix families $P_{\pm} \in C^{\infty}(\overline{\Omega_0},\C^{2n \times 2n})$ with
$$\Rep P_+ \ge CI,\quad \Rep P_-\le -CI.$$  
Again, for a solution $\mathcal V$ of \eqref{bvp}, the function  $\mW=\tilde{\mT}^{-1}\mathcal V$, $\tilde{\mT}=K\mT$,   satisfies
\begin{align}
	\label{mid1}
	\partial_x\mW_{P+}-P_{+}\mW_{P\pm}&=\bar{\mF}_{P+},\\
	\label{mid2}
	\Gamma_2 \mW&=0.
\end{align}

For a solution $\mW$, with $\mW(+\infty)=0$, we multiply \eqref{mid1} by $\pm \mW$ and integrate from $0$ to $\infty$. This time, by estimating the right hand side as
$$|\bar{\mF} \mW|_{L^1} \le \eps|\mW|_{L^2}^2+C_\eps|\bar{\mF}|_{L^2},$$
we find for $\eps$ sufficiently small,
\begin{align} 
	\label{mid3}
	|\mW_{P+}|_{L^2} ^2&\le C|\bar{\mF}_{P+}|^2_{L^2},\\
	\label{mid 4}
	|\mW_{P-}|_{L^2}^2 &\le C(|\mW_{P-}(0)|+|\bar{\mF}_{P-}|^2_{L^2}).
\end{align}
Since $\tilde{\mT}$ and its inverse are uniformly bounded, we derive immediately from Lemma \ref{lem:ef} (ii) that for $\mathcal V \in \mE_{P-}$,
$$|\Gamma_2\mW_{P-}| = |\Gamma \mathcal V| \ge C |\mathcal V| \ge C|\mW_{P-}|.$$
We can thus proceed as in the last section, namely multiply \eqref{mid3} by a sufficiently large constant and add \eqref{mid 4} to find 
$$|\mW|_{L^2}^2 \le C |\bar{\mF}|_{L^2}^2,$$
which, due to the boundedness of $\tilde{\mT}$, is equivalent to
$$|\mathcal V|_{L^2} \le C|\mF|_{L^2}.$$
By definition of the boundary value problem, this is equivalent to
$|(\hat u,\hat u_x)|_{L^2} \le |\hat f|_{L^2},$ for a solution $\hat u$ to \eqref{flt}. The proof of Proposition \ref{prop:resolvent} (ii) is now an easy consequence of Remark \ref{rem:hatU}.

\begin{proof}[Proof of Proposition \ref{prop:resolvent} (ii)]
	Let $\hat U=(\hat U^1,\hat U^2)$ satisfy \eqref{flt2}. By Remark \ref{rem:hatU},
	$\hat U^1$ solves \eqref{flt} with $\hat f$ satisfying
	$$|\hat f| \le |\hat F^1_x|+|\hat F^1|+|\hat F^2|$$
	for $|(\lambda,\eta)|$ bounded,
	and we directly conclude from the argumentation above
	$$|\hat U^1|_{H^1} \le C|\hat f|_{L^2} \le C|\hat F^1|_{H^1}+|\hat F^2|_{L^2}.$$
	Since we are now only interested in $H^s-H^s$ estimates, we can just differentiate \eqref{flt} and directly see that $\partial_x \hat u$ solves \eqref{flt} for a right hand side $\hat f=g(x,u,u_x)$ satisfying
	$$|g(u,u_x,x)| \le C(|(1+\partial_x+\partial_x^2)\hat F^1|+|(1+\partial_x)\hat F^2|+|(1+\partial_x \hat U^1)|.$$
	We obtain
	$$|\hat U^1|_{H^2} \le C|\hat F|_{H^2 \times H^1}.$$
	Inductively, we conclude
	$$|\hat U^1|_{H^k} \le C|\hat F|_{H^{k}\times H^{k-1}}.$$
	for all $k\ge 1$.
	Observing that
	$$\hat U^2=\mathcal A(\lambda \hat U^1+\hat F^1)$$
	 finishes the proof.
\end{proof}

\subsection{Large frequencies.} 
\label{subsec:lf}

It turns out that for large frequencies, it is more convenient to work with the operator $\tilde L=\tilde L_0=T L_0 T^{-1}$ defined in \eqref{tl}. Since $T$ commutes with $\chi_R(\lambda,D)$, \eqref{lfre} is equivalent to
\begin{equation}
	\label{lfre2}
	\|(\lambda-\tilde L)^{-1}\chi_R(\lambda, D)F\|_{s} \le C\|F\|_{s}
\end{equation}
for all $s\in \N$, $F \in H^s$. The crucial result is the following energy estimate.

\begin{lemma}
	\label{lem:res1}
	There exist $\delta_1$ such that for
	\begin{equation} 
		\label{small}
		|\bar u-u_-|_{W^{2,\infty}} <\delta_1
	\end{equation}
	the following holds: For all $s \in \N$, there exists $C=C(s), c=c(s)>0$ such that all solutions $V$ of 
	\begin{equation} 
		\label{semi3}
		(\lambda-\tilde L)V=F, ~~\lambda \in \C,
	\end{equation} 
	satisfy
	\begin{equation} 
		\label{res1}(\Rep \lambda+c)\|V\|^2_{s} \le C(\|F\|_{s}^2+\|V\|^2).
	\end{equation}
\end{lemma}

\begin{proof}
The proof is essentially a simpler version of the one for the nonlinear damping estimate, Proposition \ref{prop:central1} in Section \ref{s:5} and we already use some results of that section. 

W.l.o.g., let $u_-=0$, as otherwise we can consider the shifted profile $\bar u-u_-$, and coefficients $A^j(\cdot + u_-)$, $B^{jk}(\cdot+u_-)$. For a solution $V \in H^s$ of \eqref{semi3}, set $W=(V^1,\mathcal A^{-1} V^2)$. Since for all $k \in \N$, $\partial_{x}^k\mathcal A$ is bounded (on $\bar u(\R)$), the norms $\|W\|_s$ and $\|V\|_s$ are equivalent. By definition, $W$ satisfies
\begin{equation} 
	\label{lf1}
	(\lambda -\check L(\bar u))W=\tilde F,
\end{equation}
where $\tilde F=(F^1,\mathcal A^{-1}F^2)$, $\check L(\bar u)=\check L_1(\bar u)+\check L_2(\bar u)$ with
\begin{align*}
	\check L_1&=\begin{pmatrix}
		0 & \Lambda I\\
		\mathcal A^{-1}\big(\sum_{j,k=1}^dB^{jk}\partial_{x_jx_k}-\sum_{j=1}^d A^j\partial_{x_j}\big)\Lambda^{-1} & \mathcal A^{-1}\big(\sum_{j=1}^d C^j\partial_{x_j}-A^0\big)
	\end{pmatrix},\\
	\check L_2&=\begin{pmatrix}
		0 & 0\\
		\mathcal A^{-1}\left(\sum_{j=1}^d\big(B^{j1}_{x_1}+\bar B^j\big)\partial_{x_j}+\bar B^{1}_{x_1}\right)\Lambda^{-1} & \mathcal A^{-1} \big((C^1_0)_{x_1}+\bar C\big)	\end{pmatrix},\\
		\bar B^j_{lm}&=\sum_{q=1}^n
		\partial_mB^{1j}_{lq}(\bar u)(\bar u')^q,\quad \bar C_{lm}=\sum_{q=1}^n\partial_m(C^{1}_1)_{lq}(\bar u))(\bar u')^q,~~l,m=1,\ldots,n.
\end{align*}
Fix $s \in \N_0$, $\tilde F \in H^s$ and let $W \in H^s$ be a solution to \eqref{lf1}. In the following, the generic constant $C$ may depend monotonically increasing on $\|\bar u\|_s, \|W\|_s$. First, consider the Fourier-symbol of $\check L_1(\bar u)$, namely
$$
	\mL(\bar u,\xi)=\begin{pmatrix}
	0 & \lxi I\\
	\mathcal A^{-1}( \bar u)(-B(\bar u,\xi)-iA(\bar u,\xi))\lxi^{-1} & \mathcal A^{-1}( \bar u)(iC(\bar u,\xi)-A^0(\bar u))
	\end{pmatrix}.
$$ 
Let $\mU \subset \R^n$ be a $0$-neighbourhood with $\bar u(\R) \subset \mU$. By definition, $\mL$ is a symbol in $S^{1}(\bar U)$ in the sense of Definition \ref{def:sm}, and by Lemma \ref{lem:base}, $\mL(\bar u)(=\mL(\bar u(x_1),\xi))$ is a symbol in $\Gamma^{1}_l$ for all $l \in \N$ if we view $\bar u$ as a function on $\R^{d}$ constant with respect to $y=(x_2,\ldots,x_d)$. For $\alpha \in \N_0^d$ with $0\le k=|\alpha| \le s$,  we apply $\partial_x^\alpha$ to \eqref{lf1}. By product and chain rule, we have  
$$
\|[\partial_{x}^\alpha ,\check L_1(\bar u)]W\|+\|\check L_2(\bar u)\partial_{x}^\alpha W\| \le C(\|\bar u\|_{W^{1,\infty}}\|W\|_{k}+\|W\|_{k-1}),
$$
and with Lemmas \ref{lem:fm}, \ref{lem:diff}, we find
$$
	\|\Op[\mL(\bar u)-\check L_1(\bar u)\partial_x^\alpha W\| \le C \|\bar u\|_{W^{1,\infty}} \|W\|_{s}+C\|W\|_{s-1}.
$$
Thus,
\begin{equation} 
	\label{lamop}
	(\lambda-\Op[\mL(\bar u)])\partial_x^\alpha W=\partial_x^\alpha \tilde F+Q,
\end{equation} 
for $Q \in L^2$ with 
$$
	\|Q\| \le C \|\bar u\|_{W^{1,\infty}} \|W\|_{s}+C\|W\|_{s-1}.
$$
By choosing $\delta_1$ sufficiently small in \eqref{small}, we can make use of Lemma \ref{lem:kv} for $w=\bar u$ to find an operator $K(\bar u)$ satisfying \eqref{k1} and \eqref{k2}. Applying $K(\bar u)$ to \eqref{lamop} and taking the real part of the $L^2$-scalar product with $\partial_{x}^\alpha W$ leads to
$$
	(\Rep \lambda+c_2 ) \|\partial_x^\alpha W\|^2 \le C(C_\eps\|\partial_x^\alpha \tilde F\|^2+\eps\|\partial_x^\alpha W\|^2) (\|\bar u\|_{W^{1,\infty}}\|W\|_s^2+\|W\|_s\|W\|_{s-1}),
$$
for all $\eps>0$ and some $c_2>0$. Summing over $0 \le |\alpha| \le s$ and interpolating between $L^2$ and $H^s$, gives for $\eps, \delta_1$ sufficiently small and some $0<c_1<c_2$,
\begin{equation} 
	\label{lam2}
	(\Rep \lambda+c_1 ) \| W\|^2_s \le C(\|\tilde F\|_s^2+\|W\|^2).
\end{equation}
The result follows by equivalence of the norms $\|W\|_s$, $\|\tilde F\|_s$, and $\|V\|_s, \| F\|_s$.
\end{proof}

Now, estimate \eqref{lfre} is obtained in a straightforward manner  (cf. e.g. \cite{KwZ09,Kw11}).

\begin{proof}[Proof of Proposition \ref{prop:resolvent} (iii)]
	Take the  $L^2$-scalar product  of \eqref{semi3} with $ V$  to find for all $\eps>0$
   	\begin{equation}
    	\label{res2}
    	|\operatorname{Im}\lambda| \| V\|^2\le \eps \|V\|_{1}^2+C_{ \eps}\|V\|^2+C\|F\|^2.
    \end{equation}
	  Adding \eqref{res1} and \eqref{res2}, gives for $\Rep \lambda >-\theta$, $\theta$ sufficiently small,
	 \begin{equation}
	 	\label{lf2}
	 	|\lambda|\|V\|^2 +\|V\|_s ^2 \le C\|F \|_{s}^2+C_0\|V\|^2
	 \end{equation}
	 Since $\chi_R(\lambda,D)$ commutes with  $(\lambda-\tilde L)$, we can replace $V$ by $\chi_R(\lambda,D)V$ in \eqref{lf2}. By Plancherel's identity, 
	 $$ 
	 	\|\chi_R(\lambda,D)V\|_s \ge \|(1+|\eta|)\chi_R(\lambda,\eta)\mF_yV\|^2.
	 $$
 	Thus, for $R$ sufficiently large,
 	\begin{align*} 
 	|\lambda|\|V\|^2 +\|V\|_s^2 &\ge \|(1+|\lambda|^{\frac12}+|\eta|)\chi_R(\lambda,\eta)\mF_yV\|\\
 	& \ge 2C_0\|\chi_R(\lambda,\eta)\mF_yV\|^2= 2C_0\|\chi_R(\lambda,D)V\|^2,
 	\end{align*}
	and Proposition \ref{prop:resolvent} (iii) follows from \eqref{lf2}.
\end{proof}

\section{Non-linear stability}\label{s:5}

\subsection{Proof of the main result}

After having already outlined our main result in the introduction, below, we will give its precise formulation and use the methods developed in the previous sections to prove its assertions.
\begin{theo}
	\label{the:main}
	Let  $\bar u(x_1)$ be a planar shock solution to \eqref{hyp-hyp} connecting end states $u_-$ and $u_+$ satisfying (S1)-(S7). Then for $d \ge 2$,  $s \ge s_0+1$ $(s_0= [d/2]+1)$ there exist $\delta_1,\delta_2>0$ such that if 
	$$
		\|\bar u-u_-\|_{W^{2,\infty}} <\delta_1,
	$$
	for all $\phi,\psi: \R^d \to \R^n$, $\phi-\bar u \in H^{s+1} \cap W^{1,1}$, $\psi \in  H^{s} \cap L^1$ with
	\begin{equation} 
		\label{smallv}
		\|(\phi-\bar u,\psi)\|_{s+1,s}+\|(\phi- \bar u,\psi)\|_{W^{1,1}\times L^1} < \delta_2,
	\end{equation}
	there exists a unique global solution $u:[0,\infty) \times \R^d \to \R^n$ to \eqref{hyp-hyp} with $u-\bar u \in C([0,\infty),H^{s+1}) \cap C^1([0,\infty),H^{s})$, $(u(0),u_t(0))=(\phi,\psi)$, and for $t \ge 0$, 
	\begin{equation}
		\label{decayfinal1}
		\|(u(t) - \bar u,u_t(t))\|_{s+1,s} \leq C(1+t)^{-\frac{d-1}{4}}\left(\|(\phi-\bar u,\psi)\|_{s+1,s}+\|(\phi-\bar u,\psi)\|_{W^{1,1}\times L^1}\right),\\
	\end{equation}
	 and if  $s \ge s_0+3$,
	\begin{equation}
		\label{decayfinal2}
		\|(u(t) - \bar u,u_t(t))\|_{L^p} \leq C(1+t)^{-\frac{d-1}{2}(1-1/p)}\left(\|(\phi-\bar u,\psi)\|_{s+1,s}+\|(\phi-\bar u,\psi)\|_{W^{1,1}\times L^1}\right)
	\end{equation}
	for $2 \le p \le \infty$.
\end{theo}

Analogously to the treatment of the linearization, we also write the quasilinear equations \eqref{hypreg} as a first-order in time system. Here, we use variables $U=(\Lambda u,u_t)$  (instead of $(u,\mathcal A u_t)$ as  in the linear case) to obtain
\begin{equation}
\label{non-lin}	
U_t-\bar L(u)U=0
\end{equation}
with $\bar L(u)=\bar L_1(u)+\bar L_2(u)$, where
\begin{align*}
	\bar L_1(u)&=\begin{pmatrix}
		0 & \Lambda I\\
		\mathcal A^{-1}\big(\sum_{j,k=1}^dB^{jk}\partial_{x_jx_k}-\sum_{j=1}^d A^j\partial_{x_j}\big)\Lambda^{-1} & \mathcal A^{-1}\big(\sum_{j=1}^d C^j\partial_{x_j}-A^0\big)
	\end{pmatrix},\\
	\bar L_2(u)&=\begin{pmatrix}
		0 & 0\\
		\mathcal A^{-1}\sum_{j=1}^d\big(\sum_{k=1}^d  (B^{jk})_{x_k}+(C^j_1)_t\big)\partial_{x_j}\Lambda^{-1} & \mathcal A^{-1} \sum_{j=1}^d (C^j_0)_{x_j}
	\end{pmatrix}.
\end{align*}

Throughout this section, we make the assumption
\begin{equation}
	\label{smallpsi}
	|\bar{u}-u_-|_{W^{2,\infty}} < \delta_1
\end{equation}
for $\delta_1$ to be chosen sufficiently small in each step. For better readability,  we will again assume $u_-=0$ w.l.o.g.. We denote by  $v(t,x)=u(t,x)-\bar u(x_1)$ the non-linear perturbation. Then, $V=(\Lambda v,v_t)$ satisfies the Cauchy problem
\begin{align}
	\label{pert}
	&V_t-\bar L(u)V=\bar L(u)\bar U\\
	\label{periv}
	&V(0,\cdot)=V_0
\end{align}
where $\bar U=(\Lambda^{-1}\bar u,0)$, $V_0=(\Lambda(\phi-\bar{u}),\psi)$.

The crucial result of this section is to show the following damping estimate.

\begin{prop}
    	\label{prop:central1} 
    	For $s\ge s_0+1$, there exist  $\mu,\theta>0$ such that for all $T>0$, each  solution $V\in C([0,T],H^{s}) \cap C^{1}([0,T],H^{s-1})$ to \eqref{pert}, \eqref{periv}  satisfies
    	$$\|V(t)\|_s^2 \le Ce^{-\theta t}\|V(0)\|_{s}^2+C\int_0^t e^{-\theta(t-\tau)}\|V(\tau)\|^2d\tau,$$
    	provided
    	\begin{equation} 
    		\label{smallsupv}
    		\sup_{t \in [0,T]}\|V(t)\|_s \le \mu.
    	\end{equation} 
\end{prop} 

To prove this result,  consider the Fourier-symbol of $\bar L_1(u)$ (which is the same as $\check L_1(u)$ in Section \ref{subsec:lf}), namely
$$
	\mL(u,\xi)\mathbb=\begin{pmatrix}
	0 & \lxi I\\
	\mathcal A^{-1}(u)(-B(u,\xi)-iA(u,\xi))\lxi^{-1} & \mathcal A^{-1}(u)(iC(u,\xi)-A^0(u))
	\end{pmatrix}.
$$
By definition, $\mL$ is a symbol in $S^{1}(\R^n)$ in the sense of Definition \ref{def:sm}, and for $u \in W^{2,\infty}$, $\mL(u)$ is a symbol in $\Gamma^{1}_2$  by Lemma \ref{lem:base}.
In essence, the next result was shown in \cite[Section 3]{S24}. In the appendix, we will lay out how to adapt the proof in the current situation.

\begin{lemma}
	\label{lem:kv}
	Let $0$ be a stable rest state of \eqref{hyp-hyp}. Then there exist constants $c_1,c_2, C_1,C_2>0$ with the following property: For $w \in W^{2,\infty}$ with
	$$\sup_{t \in [0,T]} \|w(t)\|_{2,\infty}\le c_1,$$
	there exists a self-adjoint operator $K(w)$ of order $0$ with 
	\begin{equation} 
		\label{k1}
		1/C_1 \le K(w) \le C_1
	\end{equation}
	such that  for all $h \in L^2$,
	\begin{equation}
		\label{k2}
	\Rep ( K(w)\Op[\mL(w)]h,h) \rangle \le -c_2\|h\|^2+C_2\|h\|^2_{-1}.
	\end{equation}
	Furthermore, if $w \in C([0,T],W^{2,\infty}) \cap C^1([0,T],L^\infty)$, then $t\mapsto K(w(t))$ is continuously differentiable with
	\begin{equation} 
		\label{k3}
		\|\frac{d}{dt}K(w(t))f\| \le C\|w_t\|_{L^\infty}\|f\|.
	\end{equation}
\end{lemma}

\begin{proof}[Proof of Proposition \ref{prop:central1}]
	In the following, the generic constant $C$ may depend monotonically increasing on $\|V\|_s$ and $\|\bar u\|_{W^{s+2,\infty}}$.
	For $\alpha \in \N_0^d$ with $0\le|\alpha| \le s$, we apply $\partial_x^\alpha$ to \eqref{pert} and find 
	\begin{align} 
		\label{r}
		\partial_x^\alpha V_t&=\bar L(u)\partial_x^\alpha V+R_1\\
		\nonumber
		R_1&=\partial_x^\alpha((\bar L(u)-\bar L(\bar u)) \bar U) +[\partial_x^\alpha,\bar L(u)]V\\
		\nonumber
		&=\partial_x^\alpha((\bar L(v+\bar u)-\bar L(\bar u)) \bar U)+[\partial_x^\alpha,\bar L(v)]V+[\partial_x^\alpha,\bar L(v+\bar u)-\bar L(v)]V,
	\end{align}
	where we used $\bar L(\bar u)\bar U=0$ for $\bar u$ being a solution to \eqref{hyp-hyp}. Since all terms occurring in $\bar L(\cdot)V$/ $\bar L(\cdot)\bar U$ are of the form $M(\cdot)f$/ $M(\cdot)g$ for some smooth matrix family $M(\cdot)$ and functions satisfying 
	$$
		\|f\|_l  \le C\|V\|_{l+1},\quad \|g\|_{W^{l,\infty}} \le \|\bar u\|_{W^{l+1,\infty}},~~ l \in \N,
	$$
 chain rule and standard commutator and Moser-type estimates give
	\begin{equation} 
		\label{estr}
		\|R_1\| \le C\|V\|_s\|\bar u \|_{W^{2,\infty}}+C\|V\|_{s-1}.
	\end{equation}
	By assumption, $u=\Lambda^{-1}V^1+\bar u$, $u_t=V^2$, and thus the embedding $H^{s_0+2} \hookrightarrow W^{2,\infty}$ implies 
	\begin{equation}
		\label{uw2}
		\|u\|_{W^{2,\infty}}+\|u_t\|_{W^{1,\infty}} \le \|V\|_s+\|\bar u\|_{W^{2,\infty}}.
	\end{equation}
	Since $\bar L_2(u)V=\mathcal M(u)f_2$ for a smooth family of matrices $\mathcal M$ and a function $f_2$  with 
	$$
		|\mathcal M(u)| \le C(|\nabla u|+|u_t|), \quad \|f_2\|_l \le \|V\|_l,~~l \in \N,
	$$
	we conclude again with Moser-estimates
	\begin{equation} 
		\label{l2}
		\|\bar L_2(u)\partial_x^\alpha V\| \le C\|V\|_s(\|V\|_s+\|\bar u\|_{W^{1,\infty}}).
	\end{equation}
	Due to \eqref{uw2}, $\Op[\mL(u)]$ is well-defined and Lemma \ref{lem:diff} gives
	\begin{equation} 
		\label{diff}
		\|(\bar L_1(u)-\Op[\mL(u)])\partial_x^\alpha V\| \le C(\|u\|_{W^{1,\infty}}\|\partial_x^\alpha V\|+\|\partial_x^\alpha V\|_{-1}) \le C\|V\|_s(\|V\|_s+\|\bar u\|_{W^{1,\infty}})+C\|V\|_{s-1})
	\end{equation}
	Combining \eqref{r}, \eqref{estr}, \eqref{l2}, \eqref{diff} yields
	\begin{align} 
		\label{eqr3}
		\partial_x^\alpha V_t&=\Op[\mL(u)]+R_2,\\ 
		\label{r3}
		\|R_2\| &\le C\|V\|_s(\|V\|_s+\|\bar u\|_{W^{1,\infty}})+C\|V\|_{s-1})
	\end{align}
	Choosing $\delta_1$ sufficiently small, assuming \eqref{smallsupv} with $\mu$ small and using \eqref{uw2}, we can use Lemma \ref{lem:kv} with $w=u$. We apply the resulting operator $K(u)=K(u(t))$  to \eqref{eqr3} and take the $L^2$-scalar-product with $\partial_x^\alpha V$ to obtain
	\begin{align*}\frac12\frac{d}{dt} (K(u)\partial_x^\alpha V,\partial_{x}^\alpha V) &=\Rep(K(u)\Op[\mL(u)]\partial_x^\alpha V,\partial_x^\alpha V)\\
		&\quad+\Rep((K(u))_t\partial_x^\alpha V,\partial_x^\alpha V)+\Rep(R_2,\partial_x^\alpha V,\partial_x^\alpha V)\\
	&\le -c_1\|\partial_x^\alpha V\|^2+C\|u_t\|_{L^\infty}\|\partial_x^\alpha V\|^2+C\|V\|_s^2(\|V\|_s+\|\bar u\|_{W^{1,\infty}})\\
	& \quad +C\|V\|_{s-1}\|V\|_s.
	\end{align*}
	Lastly, we use $u_t=V^2$ and the interpolation inequality
	$$
		\|V\|_{s-1} \le C(\|V\|_s+\|V\|)
	$$ 
	to find a $c>0$ such that for $\|V\|_s+\|\bar u\|_{W^{2,\infty}}$ small,
	$$\frac12\frac{d}{dt} (K(u)\partial_x^\alpha V,\partial_{x}^\alpha V) \le -c\|V\|_s^2+C\|V\|^2,$$
	and the result follows from the positive definiteness of $K(u)$ and  Gronwall's inequality.
\end{proof}

As mentioned in the introduction, Theorem \ref{the:main} can now be shown by a standard non linear iteration scheme \cite{Z04,Z07} and we conclude this section by recalling the crucial steps. 

\begin{proof}[Proof of Theorem \ref{the:main}]
	 Let $s \ge s_0+1$, $(\phi,\psi): \R^d \to \R^n$ with 
	 $$
	 	(\phi-\bar u,\psi) \in (W^{1,1} \cap H^{s+1}) \times (L^1 \cap H^s)
	 $$  
	 satisfying \eqref{smallv} for $\delta_2$ to be chosen. Let $u$ be a solution of \eqref{hypreg} such that $u(0,\cdot)=\phi$, $u_t(0,\cdot)=\psi$.
	  Let $T>0$ be the largest constant such that the solution $u$ to  \eqref{hypreg} satisfies
	  $v=(u-\bar u) \in C^l([0,T),H^{s+1-l})$, $l=0,\ldots, s$, and additionally 
	  $$
	  	\|(v(t),v_t(t))\|_{s+1,s} \le \mu, ~t \in [0,T),
	  $$
	  with $\mu$ as in Proposition \ref{prop:central1}. Note that such a $T>0$ exists, due to local well-posedness, Proposition \ref{prop:locwell}, provided we choose $\delta_2<\mu$ in \eqref{smallv}. By definition, $W=(v,\mathcal A(\bar u)v_t)$ satisfies
	$$
		W_t-LW=\partial_x Q^1(v,v_t,\partial_x v)+\partial_t Q^2(v,v_t,\partial_x v)=:Q,
	$$
	with $L$ the linear operator defined in \eqref{l}, and for $l=1,2$,
	\begin{align*} 
		Q^l(v,v_t,\partial_x v)= (0,\sum_{j=0}^dq^l_j(v)v_{x_j}+q^l_{d+1}(v)v),~~x_0=t,
	\end{align*}
	for smooth functions $q^l_j$ with $q^l_j(0)=0$. 
	
	 Thus, for $(v,v_t)$ uniformly bounded in $H^{k+1} \times H^k$, $k \ge s_0$, we find by Moser inequalities 
	\begin{align}
		\label{q1}
		\|Q^l(v,v_t,\partial_x v)\|_{k} +\|Q^l(v,v_t,\partial_x v)\|_{W^{1,1}} \le C\|W\|_{k+1,k}^2.
	\end{align}
 	Next, define
 	$$
 		\zeta(t)=\sup_{0 \le \tau \le t}\|W(\tau)\|_{1,0}(1+\tau)^{\frac{d-1}{4}},~~ t \in [0,T],
 	$$
 	$W_0=W(0)=(v(0),\mathcal A(\bar u)v_t(0)))$, and
 	$$
 		|\|W_0\||=\|W_0\|_{W^{1,1} \times L^1}+\|W_0\|_{s+1,s}.
 	$$
 	For $V=(\Lambda v,v_t)$, $\|W\|_{l+1,l}$ is equivalent to $\|V\|_{l}$, $0 \le l \le s$, and hence we get from Proposition \ref{prop:central1}
  		\begin{align}
  			\label{v2}
 		\|W(t)\|_{s+1,s}^2 &\le Ce^{-\theta t}\|W_0\|_{s+1,s}^2+\int_0^t e^{-c\theta(t-\tau)}\|W(\tau)\|^2_{1,0}d\tau\\
 		\nonumber
 		& \le C(\|W_0\|^2_{s+1,s}+\zeta(t)^2)(1+t)^{-\frac{d-1}{2}}.
 		\end{align}
 	Now, Corollary \ref{decacy:lin} and equations \eqref{q1}, \eqref{v2} yield for some $c>0$
 	\begin{align*}
 		\int_0^t e^{-\theta(t-\tau)} \|W(\tau)\|_{1,0}^2 d\tau &\le C(1+t)^{-\frac{d-1}{2}}(|\|W_0\||^2+\|Q^2(0)\|^2_{H^2 \cap L^1})\\
	 	&\quad+C \int_0^t e^{-c(t-\tau)}(\|\partial_xQ^1\|_2^2+\|Q^2\|_3^2)d\tau\\
	 	& \quad+ C \left(\int_0^t (1+t-\tau)^{-\frac{d}{4}-\frac12}
	 	(\|Q^1\|_{L^1}+\|Q^2\|_{L^1})d\tau\right)^2\\
	 	&\le C(1+t)^{-\frac{d-1}2}|\|W_0\||^2+C\int_0^te^{-c(t-\tau)}\|W(\tau)\|_s^4d\tau\\
	 	&\quad +C\left(\int_{0}^t (1+t-\tau)^{-\frac{d}{4}-\frac12} \|W(\tau)\|_s^2 d\tau\right)^2,\\
	 	&\le C (1+t)^{-\frac{d-1}2}(|\|W_0\||^2+\zeta(t)^4),
 	\end{align*}
 	where for the last estimate, $d \ge 2$ is crucial. Using again \eqref{v2}, this implies
 	\begin{equation} 
 		\label{zeta}
 	\|W(t)\|_{s+1,s} \le C_0(1+t)^{-\frac{d-1}{4}}(|\|W_0\||+\zeta(t)^2).
 	\end{equation}
 	for some $C_0>0$.  Since $\zeta$ is monotonically increasing, taking the supremum of both sides in \eqref{zeta}, gives
 	$$
 	\zeta(t) \le C_2(|\|W_0\||+\zeta(t)^2), ~0 \le t \le T,
 	$$	
 	which implies by continuity of $\zeta$,
 	\begin{equation}
 		\label{w0}
 		\zeta(t) \le C_2|\|W_0\||
 	\end{equation}
 	for some $C_2>0$ and $|\|W_0\|| < \tilde\delta$, $\tilde \delta=\tilde\delta(C_2)$ sufficiently small. Due to \eqref{zeta} this gives \eqref{decayfinal1}. In addition,
 	\begin{align*} 
 		\|(v(t),v_t(t))\|_{s+1,s}&\le \max\{1,\max_{0 \le k \le s}\sup_{x \in \R}\partial_{x_1}^k(\mathcal A(\bar u(x))^{-1})\}\|W(t)\|_{s+1,s} \\
 		&\le  \max\{1,\max_{0 \le k \le s}\sup_{x \in \R}\partial_{x_1}^k(\mathcal A(\bar u(x))^{-1})\}(C_0+C_2^2)|\|W_0\||=:\tilde C_3|\|W_0\||\\
 		&=\tilde C_3(\|(\phi-\bar u,\mathcal A(\bar u)\psi)\|_{s+1,s}+\|(\phi- \bar u,\mathcal A(\bar u)\psi)\|_{W^{1,1}\times L^1})\\
 		&\le C_3(\|(\phi-\bar u,\psi)\|_{s+1,s}+\|(\phi- \bar u,\psi)\|_{W^{1,1}\times L^1})
 	\end{align*}
 	for $t \in [0,T)$, $|\|W_0\||< \tilde \delta$. Thus, if $\delta_2 < \min\{\tilde \delta,\mu/(2C_3),\}$ in \eqref{smallv}, we obtain
 	$$
 		\sup_{t \in [0,T)}\|v(t),v_t(t)\|_s < \mu/2.
 	$$
 	Since $T$ was chosen as the largest constant $\tilde T$ with
 	$$
 		\sup_{t \in [0,\tilde T)}\|v(t),v_t(t)\|_s < \mu,
 	$$
 	and $t \mapsto \|v(t),v_t(t)\|_s$ is continuous, the above implies $T=\infty$ and we have also shown global existence.

 	It remains to show the $L^p$-estimate \eqref{decayfinal2}. To this end  define 
 	$$
 		\zeta_2(t)=\sup_{0 \le \tau \le t}(\|W\|_{L^\infty}(1+\tau)^{\frac{d-1}{2}}+\|W\|_{1,0}(1+\tau)^{\frac{d-1}{4}}),~ t \in [0,T].
 	$$
 	From Propositions  \ref{prop:dl1}, \ref{prop:dect}, and equation \eqref{zeta}, we obtain for $s \ge s_0+3$ the $L^\infty$-estimate
 	\begin{align*} 
 		\|W(t)\|_{L^\infty} &\le \|S(t)W_0\|_{L^\infty}+\left\|\int_0^t S(t-\tau)(\partial_xQ^1(\tau)+\partial_\tau Q^2(\tau))d\tau\right\|_{L^\infty}\\
 		& \le C(|\|W_0\||(1+t)^{-\frac{d-1}{2}}+\|Q^2(t)\|_{L^\infty}+(1+t)^{-\frac{d-1}2}\|Q^2(0)\|_{ H^{s_0+2}} \cap L^1)\\
 		&\quad +C\int_0^t e^{-c(t-\tau)} \|\partial_{x}Q^1\|_{s_0+1}+\|Q^2\|_{s_0+3}d\tau +(1+t-\tau)^{-d/2}\|Q^1\|_{L^1}+\|Q^2\|_{L^1}d\tau\\
 		&\le C(|\|W_0\||+|\|W_0\||^2+\zeta(t)^2)(1+t)^{-\frac{d-1}{2}}+C\int_0^t (1+t-\tau)^{-\frac{d}{2}}\|W(\tau)\|_{s+1,s}^2d\tau\\
 		&\le C(|\|W_0\||+|\|W_0\||^2+\zeta(t)^2)(1+t)^{-\frac{d-1}{2}}\\
 		& \quad +C(|\|W_0\||^2+\zeta(t)^2)\int_0^t (1+t-\tau)^{-\frac{d}{2}} (1+\tau)^{-\frac{d-1}{2}}d\tau\\
 		&\le C_1(1+t)^{-\frac{d-1}{2}}(|\|W_0\||+|\|W_0\||^2+\zeta(t)^2).
 	\end{align*}
 	Adding this to \eqref{zeta}, we get 
 	$$
 		\zeta_2(t) \le C(|\|W_0\||+\zeta_2(t)^2),
 	$$	
 	and thus by continuity of $\zeta_2$,
 	\begin{equation}
 		\label{zeta2}
 		\zeta_2(t) \le C|\|W_0\||.
 	\end{equation}
 	By definition of $\zeta_2$ and $|\|W_0\||$, \eqref{zeta2} implies in particular \eqref{decayfinal2} for $p=\infty$. As we have already shown \eqref{decayfinal1}, \eqref{decayfinal2} it holds for all $p \in [2,\infty]$ by interpolation.
\end{proof}

\subsection{Cases without (S5)} \label{s:5.2}

In \cite{N09} Nguyen was able to remove condition (S5) in the hyperbolic-parabolic case at the cost of a worse decay rate. In this section, we point out that his results also apply in our case. As an aside, we also do not need to assume (S6). As stated in the introduction, assumption (S5) is automatic in $2$ dimensions. For $d\ge 3$, we have the following theorem. 
\begin{theo}
	\label{the:S5} 
	Under the assumptions of Theorem \ref{the:main}, excluding (S5), (S6), there exists a unique global solution $u \in C([0,\infty),H^{s+1}) \cap C^1([0,\infty),H^{s})$ to \eqref{hyp-hyp} with $(u(0),u_t(0))=(\phi,\psi)$ and, for $t \ge 0$, $2 \le p \le \infty$,
	$$
		\begin{aligned}
			\|u(t) - \bar u\|_{s+1}+\|u_t(t)\|_s &\leq C(1+t)^{-\frac{d-2}{4}}(\|\phi-\bar u\|_{s+1}+\|\psi\|_{s}+\|(\phi-\bar u,\psi)\|_{W^{1,1}\times L^1}),\\
			\|(u(t) - \bar u,u_t(t))\|_{L^p} &\leq C(1+t)^{-\frac{d-1}{2}(1-1/p)+\frac14}(\|\phi-\bar u\|_{s+1}+\|\psi\|_{s}+\|(\phi-\bar u,\psi)\|_{W^{1,1}\times L^1}),
		\end{aligned}
	$$
\end{theo}

(S5), (S6) only enter in the estimates for the linearization in the low frequency regime. Thus, Theorem \ref{the:S5} follows with the same arguments as before once we have shown
\begin{equation}
	\label{s12}
	\|S_1(t)\partial_{x}^\tau F\|_{L^p} \le C(1+t)^{-\frac{d-1}{2}(1-1/p)+\frac14-\frac{|\tau|}{2}}(\|F^1\|_{W^{1,1+|\tau|}}+\|F^2\|_{L^1}),~~|\tau|=0,1.
\end{equation}
 As with estimate \eqref{s1} in Proposition \ref{prop:dl1}, \eqref{s12} follows from the resolvent estimate below (cf.\cite{N09}).

\begin{lemma}
	\label{lem:resS5}
	There exists $r>0$  such that for all $(\lambda,\eta) \in M_c$ with $|(\lambda,\eta)|\le r$ for all $2 \le p \le \infty$, $k=0,1$
	\begin{equation} 
		\label{sfreS5}
		|(\lambda-L_\eta)^{-1}\partial_{x_1}^k\hat F|_{L^p}  \le
		C|(\lambda,\eta)|^{-\frac32+k}(|\hat F^1|_{W^{1,1+k}_{x_1}}
		+|\hat F^2|_{L^1_{x_1}}),
	\end{equation}
\end{lemma}

\begin{proof}[Proof of \eqref{s12}]
	Assuming we have already shown Lemma \ref{lem:resS5}, we can estimate
	\begin{align*}
			\|(\lambda-L_\eta)^{-1}\partial_x^k\hat F\|_{L^\infty_\eta L^p_{x_1}} &\le C\rho^{-\frac32+k}(\|\hat F^1\|_{L^\infty_\eta W^{1+k}_{x_1}}+\|\hat F^2\|_{L^\infty_\eta L^1_{x_1}}\\
			& \le C(\|F^1\|_{W^{1+k}}+\|F^2\|_{L^1})\rho^{-\frac32+k}.
	\end{align*}
	by the Hausdorff-Young inequality. Now, we can repeat the proof in Section 2.4. of \cite{N09} word for word after replacing $\|f\|_{L^1}$ by $\|F^1\|_{W^{1,1}}+\|F^2\|_{L^1}$ or by $\|F^1\|_{W^{1,2}}+\|F^2\|_{L^1}$ for derivative estimates.
\end{proof}

To show Lemma \ref{lem:S5}, we again use the equivalence of \eqref{flt} with the boundary value problem \eqref{bvp}. (S5), (S6) are only used to uniformly diagonalize the glancing blocks in Proposition \ref{prop:dec}. Without (S5), this can not be done. Instead we show the following weaker result. 

\begin{lemma}
	\label{lem:S5}
	For each $\underline {\hat \zeta} \in S^d_+$, there exists a neighbourhood of $(0,\underline{ \hat \zeta})$, and for each glancing block $Q\in \C^{ m\bar s\times m\bar s}$, a smooth basis transformation $T \in C^\infty(\Omega_0,\operatorname{Gl}_{m\bar s})$ such that
	$$(T^{-1}QT)(\rho,\hat \zeta)=\begin{pmatrix}
			q(\hat \zeta) & \cdots & 0 \\
			\vdots & \cdots & 0\\
			0 & \cdots & q(\hat \zeta)
		\end{pmatrix}+\rho\begin{pmatrix}
		N^{11}(\rho, \hat \zeta) & \cdots & N^{1m}(\rho,\hat \zeta)\\
		\vdots & \cdots & \vdots\\
		N^{ m 1}(\rho, \hat \zeta) &  \cdots & N^{ m m}(\rho \hat \zeta)\end{pmatrix}.$$
with $q(\hat \zeta), N^{jk} \in \C^{ \bar{s} \times \bar s}$ satisfying the following properties:
\begin{enumerate}
	\item[(i)] $q(\tau,\gamma,\eta)$ has purely imaginary coefficients for $\gamma =0$ and $q(\underline{\hat \zeta})=i(\underline \xi_1I_\nu+J)$ for some $\underline \xi_1 \in \R$, where $J$ is the Jordan block of size $\bar s$. 
	\item[(ii)] For some $n_{jk} \in \C$,
	$$N^{jk}(0,\underline{\hat \zeta})=\begin{pmatrix}
		\ast & 0 & \ldots & 0\\
		\vdots & 0 & \ldots & 0 \\
		n_{jk} & 0 & \ldots & 0
	\end{pmatrix}.$$
\item[(iii)] For
$$N^{\flat}=\begin{pmatrix}
	n_{11} & \cdots & n_{1m}\\
	\vdots & \cdots & \vdots\\
	n_{m1} & \cdots & n_{m m}
\end{pmatrix},$$ 
we find $\operatorname{Re}(\partial_{\gamma}q_{1\bar s}(\underline{ \hat \zeta}) N^{\flat}) >0$.
\end{enumerate} 
\end{lemma}

\begin{proof}
	We closely follow \cite[proof of Lemma 2.10]{MZ05}.
	As in the proof of Proposition \ref{prop:dec}, consider a glancing block arising from an eigenvalue $i\underline \xi_1$ of $H_{0+}(\underline{\hat \tau},0,\underline{\hat \eta})$ and suppress the index $+$. The properties of $Q(0,\hat \zeta)$ are of course solely determined by the first-order symbol and were shown in \cite{M00}. In particular, we find a first family of (smooth uniformly bounded) basis transformations such that written in this basis
	$$Q(\rho,\hat \zeta)=Q_0(\hat \zeta)+\rho Q_1(\rho, \hat \zeta),$$
	where $Q_0=\operatorname{diag}(q,\ldots,q)$, $q$ satisfying property (i). 
	
	To show condition (iii), we can just repeat the arguments in the proof of Proposition \ref{prop:dec} restricted to the base point $(\xi_1(\hat \eta),\hat \eta)=(\underline\xi_1,\underline{\hat \eta})=:\underline \xi$. This gives
	\begin{equation} 
		\label{qgamma}
		 \partial_\gamma  q_{1\bar s}(\underline{\hat \zeta})=-(\partial_{\xi_1}^{\bar s} a)(\underline \xi)^{-1})=:-p^{-1} \in \R\setminus\{0\},
	\end{equation}
	and the spectrum of $N^{\flat}$ is that of $p^{-1}L(\underline \xi)\check A^1H_1(\underline{\hat \tau},0,\underline{\hat \eta})R(\underline \xi)$, where $R(\underline \xi)$ is a right-projector on the kernel of $\underline{\hat \tau}+A(\underline \xi)$. Thus, by Lemma \ref{lem:h0}, the spectrum of $\partial_\gamma  q(\underline{\hat \zeta})^{1\bar s}N^\flat$ is contained in the open right half-plane. 
	
	From this point, to construct a further basis transformation still satisfying (i) and, additionally, (ii), (iii), we can cite \cite{MZ05} word for word.
\end{proof}

\begin{proof}[Proof of Lemma \ref{lem:resS5}]
Analogously to the results in Section \ref{s:4}, near each base point $(0,\underline{\hat \zeta})$, we can construct a smooth family of basis-transformations $T$, such that
$$T^{-1}\mG(\infty)T=\begin{pmatrix}
	P_+ & 0 & 0\\
	0 & P_- & 0\\
	0 & 0 & \rho \tilde H
\end{pmatrix},\quad \tilde H=\operatorname{diag}(Q_1,\ldots,Q_{\bar k}),$$
where $\pm\Rep P_{\pm} \ge CI$ and $Q_k$ satisfies $\Rep Q_k >C(\rho+\hat \gamma)I$, $\Rep Q_k <-C(\rho+\hat \gamma)$, or $Q_k$ is a glancing block satisfying (i)-(iii) of Lemma \ref{lem:S5}. Under this structural condition together with the boundary estimates of Lemma \ref{lem:ef}, Nguyen \cite{N09} showed that for $\zeta \in M_c$, $|\zeta|$ small, the $L^p$-solution $\mathcal V$ to \eqref{bvp} satisfies 
$$
	|\mathcal V|_{L^p} \le C\rho^{-\frac32}|\mF|_{L^1},~~p\in [2,\infty].
$$
From this, the assertion  follows as in the proofs of Corollary \ref{coro:aux} and Proposition \ref{prop:resolvent} (i) with $\beta$ replaced by $\rho^{-\frac12}$. 
\end{proof}

\subsubsection*{Acknowledgements} Thanks to Heinrich Freist\"uhler, Johannes B\"arlin, and Valentin Pellhammer
for numerous helpful discussions. This work was initiated during and greatly aided by two visits
of the second author to University of Konstanz, the first funded by the DMR network and the second
by the DMR and Humboldt foundation; he gratefully thanks the University of Konstanz and Humboldt
foundation for their hospitality and support.
\bigskip

\textbf{Declarations}
\smallskip

\textbf{Competing interests.} The authors have no competing interests to declare that are relevant to the content of this article.

\smallskip
\textbf{Data availability statement.} No data has been used in this work.

\appendix

\section{Para-differential operators}\label{s:paradiff}

\subsection{Central notions}

In this section, we recall the central definitions and results on para-differential operators used in the present work.  Para-differential operators were first introduced by Bony \cite{Bo81}, Meyer \cite{Me81}, and further developed by H\"ormander \cite{Ho97}. For a concise presentation, consult e.g. \cite{BS06,M08}. This paragraph closely follows the presentation in \cite{S24}, where also the refined G\aa{}rding inequality Proposition \ref{prop:garding} is proven.
\begin{defi}
	We define $\Gamma^m_k$ to be the vector-space of functions $A: \R^d \times \R^d \mapsto \C^{n\times n}$ such that,
	\begin{enumerate}
		\item[(i)]
		for almost all $x \in \R^d$, the mapping $\xi \mapsto A(x,\xi)$ is in $C^\infty(\R^d,\C^{n\times n})$,
		\item[(ii)]
		for any $\alpha \in \N_0^d$ and $\xi \in \R^d$, the mapping $x  \mapsto \partial_\xi^\alpha A(x,\xi)$ belongs
		to $W^{k,\infty}(\R^d,\C^{n \times n})$ and there exists $C_\alpha > 0$ not depending on $\xi$ such that
		\begin{equation}
			\label{eq:semingamma}
			\|\partial^\alpha_\xi A(\cdot,\xi)\|_{L^{\infty}} \le C_\alpha \lxi^{m-|\alpha|},
		\end{equation}
		\item[(iii)]
		with semi-norms 
		$$\|A\|_{m,k,L}=\max_{\alpha  \le L} \sup_{\xi \in \R^d}\lxi^{|\alpha|-m}\|\partial^\alpha_\xi A(\cdot,\xi)\|_{W^{k,\infty}},$$
		$\Gamma^m_k$ is a Fr\'{e}chet space.
	\end{enumerate}
\end{defi}

For a symbol $A \in \Gamma_k^m$, we denote by $\Op[A]$ the para-differential operator associated to $A$.  It is defined on the Schwartz-space by 
$$(\Op[A]\phi)(x)=\frac{1}{(2\pi)^{-\frac{d}{2}}}\int_{\R^d}e^{ix\xi}(\sigma(A))(x,\xi)\mF_\xi \phi d\xi,$$
where $\sigma$ is a suitable smoothing operator for symbols in $\Gamma^{m}_k$ (for details see e.g. [M09]). For $A \in \Gamma^m_k$ and $l \in \R$, $\Op[A]$ can be extended to a bounded linear operator $\Op[A]:H^{l+m} \to H^{l}$, $l \in \R$ (an operator of order $m$).   Furthermore, $A\mapsto \Op[A]$ is a continuous linear operator between locally convex spaces, i.e., we have for some $C>0$, $L=L(m,l) \in\ N_0$,
\begin{equation}
	\label{opa}
   	\|\Op[A]f\|_l \le C\|A\|_{m,0,L}\|f\|_{m+l}.
\end{equation}

Note that due to application of the smoothing operator $\sigma$, for $A \in \Gamma^m_k$ independent of $x$, $\Op[A]$ does not quite coincide with  $\mF^{-1}A\mF$. However the following holds.

\begin{lemma}
	\label{lem:fm}
	Let $A \in \Gamma^{m,k}$ does not depend on $x \in \R^d$. Then,
	the operator $\Op[A]-\mF^{-1}A\mF$ is infinitely smoothing, i. e., for any $r,l \in \R$, there exists $C_{rl}$ such that
	$$
		\|(\Op[A]-\mF^{-1}A\mF)f\|_l \le C_{rl}\|f\|_r,~~f \in H^r.
	$$
	In addition, $\Op[A]^*=\Op[A^*]$.
\end{lemma}

In the special case that $A \in \Gamma^m_k$ does not depend on $\xi$, i.e. $A \in W^{k,\infty}$, the following holds (\cite[Corollary C.4]{BS06}).

\begin{lemma}
	\label{lem:diff}
For $j=1,\ldots,d$, $A \in W^{1,\infty}$
$$
	\|(\Op[A]-A)\partial_{x_j} f\| \le C\|A\|_{W^{1,\infty}}\|f\|+c(\|A\|_{L^\infty})\|f\|_{-1},~~A \in W^{1,\infty},~f \in L^2,$$
	where $c(\|A\|_{L^\infty})$ can be chosen monotonically increasing.
\end{lemma}

In addition, we need the following estimates on para-differential operators.

\begin{prop}
	\label{prop:adj}
	Let $A \in \Gamma_1^m$. Then  $\Op[A^*]-\Op[A]^*$ is an operator of order $m-1$ and
	$$\|(\Op[A]^*-\Op[A]^*)f\|_{l} \le C\|\partial_x A \|_{m,0,L}\|f\|_{l+m-1},~~l\in \R$$
	where $L \in \N$, $C>0$ only depend on $l$ and $m$.
\end{prop}

\begin{prop}
	\label{prop:product}
	Let $A \in \Gamma^m_1$, $B \in \Gamma^{\tilde m}_1$, $m,\tilde m \in \R$.  Then  $\Op[B]\Op[A]-\Op[BA]$ is an operator of order $m+\tilde m-1$, and for $l \in \R$,
	$$
		\|(\Op[B]\Op[A]-\Op[BA])f\|_l \le C(\|B\|_{\tilde m,0,L}\|\partial_x A \|_{m,0,L}+\|\partial_x B\|_{\tilde m,0,L}\|A\|_{m,0,L})\|f\|_{l+m+\tilde m-1},
	$$
	where $L \in \N$, $C>0$ only depend on $l,m$ and $\tilde m$.\\
	If the symbol $A$ dose not depend on $x$,  then $\Op[BA]=\Op[B]\mF^{-1}A\mF$.
\end{prop}

In Section \ref{s:5}, the results above are applied to symbols of the form $(x,\xi)\mapsto F(u(x),\xi)$, where $F \in C^\infty(\mU \times \R^d, \C^{n\times n})$ ($\mU \subset \R^n$ some $0$-neighbourhood) and $u \in W^{k,\infty}(\R^d,\R^n)$ for some $k \in \N$. In the following, let $\mU \subset \R^N$ be a $0$-neighbourhood.

\begin{defi}
	\label{def:sm}
	We denote by $S^m(\mU):=S^m(\mU,\C^{n \times n})$ the set of all functions $F \in C^\infty(\mU \times \R^d, \C^{n \times n})$ for which, for any $\alpha,\beta \in \N_0^d$, there exists $C_{\alpha\beta}>0$ such that for all $(u,\xi) \in \mU \times \R^d$,
	\begin{equation} 
		\label{eq:help*}
		|\partial_u^\beta\partial_\xi^\alpha F(u,\xi)| \le C_{\alpha\beta}\lxi^{m-|\alpha|}.
	\end{equation}
\end{defi}

For functions $F:\mU \times \R^d \to \C^{n\times n}$ and $u: \R^d \to \mU$, with a slight abuse of notation, we write $F(u)$ for the composition
$$F(u): \R^d \times \R^d \to \C^{n \times n}, (x,\xi) \mapsto F(u(x),\xi).$$
The results below follow directly by Moser inequalities, \eqref{opa}, and Propositions \ref{prop:adj}, \ref{prop:product}.

\begin{lemma}
	\label{lem:base}
	Let $F \in S^m(\mU)$ and $u \in W^{k,\infty}$ with $\|u\|_{k,\infty} \le \mu$. Then $F(u) \in \Gamma^m_k$ and
	\begin{align*} 
		\|F(u)\|_{m,k,L} &\le  C(m,k,\mu,F),\\
		\|F(u)-F(0)\|_{m,k,L} &\le C(m,k,\mu,F)\|u\|_{W^{k,\infty}}.
	\end{align*}
\end{lemma}

\begin{prop}
	\label{prop:central}
	Let $F \in S^m(\mU), G \in S^{\tilde m}(\mU)$ and $\mu \in \R$. Then for all $l \in \R$, there exist constants $C_1=C_1(m,l,F)$, $C_2=C_2(m,\tilde m, F,G)$ such that for all $u \in W^{1,\infty}$ with $\|u\|_{W^{1,\infty}} \le \mu$,
	\begin{align}
		\label{adj}
		\|(\Op[F(u)]^*-\Op[F(u)^*])f\|_l &\le C_1\|\nabla u\|_{L^\infty}\|f\|_{m+l-1},\\
		\label{comp}
		\|(\Op[G(u)]\Op[F(u)]-\Op[G(u)F(u)])f\|_l& \le C\|\nabla u\|_{L^\infty}\|f\|_{m+\mu+l-1}.
	\end{align}
	Furthermore, if $u \in C^1([0,T],L^\infty)$, the function $t \mapsto \Op[F(u(t))]$ is continuously differentiable with 
	$$
		\frac{d}{dt}\Op[F(u(t))]=\Op[F(u_t(t))], 
	$$
	and in particular,
	\begin{equation} 
		\label{deriv}
		\|\frac{d}{dt}\Op[F(u(t))]f\|_l \le C\|u_t\|_{L^\infty}\|f\|_{m+l}.
	\end{equation}
\end{prop}

The crucial result for proving the nonlinear energy estimates and obtaining decay-estimates in the large frequency region is the following.

\begin{prop}[\cite{S24}, Proposition 2.24]
	\label{prop:garding}
	Let $u \in W^{2,\infty}$ with $\|u\|_{W^{2,\infty}} \le \mu$, and $F \in S^m(\mU)$ with $F(v,\xi)+F(v,\xi)^* \ge 0$ for all $v \in \mU$ and $|\xi| \ge |R|$, for some $R\ge 0$. Then there exists $C=C(m,\mu,F),c=(m,\mu,F)>0$ such that for all $v \in \mS(\R^d,\C^n)$,
	$$\langle (\Op[F_u]+\Op[F_u]^*)v,v\rangle_0 \ge -C\|u\|_{W^{2,\infty}}^{\frac12}\|v\|^2_{(m-1)/2}-c\|v\|_{(m-1)/2-1}^2.$$
\end{prop}

\begin{proof}
	In \cite{S24}, the Proposition is formulated and proven under the assumption $u \in H^{s+2}$ for $s>d/2$ with $\|u\|_{W^{2,\infty}}$ replaced by $\|u\|_{s+2}$ norms. Since by Sobolev embedding, 
	$$
		\|u\|_{W^{2,\infty}} \le C \|u\|_{s+2}
	$$
	for $s>d/2$, the assumption in \cite{S24} is more restrictive. However, in \cite{S24}, analogues  of Lemma \ref{lem:base} and Proposition \ref{prop:central} were also formulated in terms of $\|u\|_{s}$ norms for $s$ suitably large. Using instead Lemma \ref{lem:base}, Proposition \ref{prop:central} as stated above, allows us to repeat the proof in \cite{S24} word for word except for  replacing $\|u\|_{s+2}$ by $\|u\|_{W^{2,\infty}}$ everywhere. 
\end{proof}

\subsection{Proof of Lemma \ref{lem:kv}}

First, we need the following crucial consequence of condition (D2) (\cite[Proposition 3.7]{S24}). 

\begin{lemma}
	\label{lem:d3}
	Let $u_* \in \R^n$ satisfy (D3). Then, there exists a neighbourhood $\mU \subset \R^n$ of $u_*$, $r>0, c_*>0$, and  
	a  mapping $\mD\in C^\infty(\Omega,\C^{2n \times 2n})$, $\Omega:=\bar \mU \times \{\xi \in \R^d:|\xi| \ge r\}$ such that:
	\begin{enumerate} 
		\vspace{-0.5\baselineskip}
		\item[(i)] For all $(u,\xi) \in \Omega$, 
		\begin{equation} 
			\label{sym}
			\mD(u,\xi)=\mD(u,\xi)^* \ge c_* I_n,
		\end{equation}
		and  
		\begin{equation} 
			\label{sym2}
			\mD(u,\xi)\mL(u,\xi)+(\mD(u,\xi)\mL(u,\xi))^* \le -c_* I_{2n}.
		\end{equation}
		\item[(ii)] For any $\alpha,\beta \in \N_0^d$, there exist $C_{\alpha\beta}>0$ with
		\begin{equation}
			\label{sym1}
			|\partial_u^\beta \partial_\xi^\alpha \mD (u,\xi)| \le C_{\alpha\beta}\lxi^{-|\alpha|},~~(u,\xi) \in  \Omega.
		\end{equation}
	\end{enumerate}
\end{lemma}

Now, we apply this to the rest state $u_*=0$. First, extend $\mD:\Omega \times \{|\xi| \ge r\} \to \C^{2n \times 2n}$ trivially by zero to a obtain a function defined on $\Omega \times \R^d$. Choose $\phi \in C^\infty(\R^d)$ with $0 \le \phi \le 1$, $\phi(\xi)=0$ for $|\xi| \le 2r$, $\phi(\xi)=1$ for $|\xi| \ge 3r$ and set 
$$
	\mD_1(w,\xi):=\phi(\xi)\mD(w,\xi),\quad (w,\xi) \in \bar\mU \times \R^d.
$$
Choose another function $\psi \in C^\infty(\R^d)$ with $0 \le \psi \le 1$,  $\psi(\xi)=0$ for $|\xi| \ge 5r$, $\psi(\xi)=1$ for $|\xi| \le 4r$ and define
$$
	\mD_2(w,\xi)=\mD_1(w,\xi)+\psi(\xi)I_{2n},~~(v,\xi) \in \bar \mU \times \R^d.
$$
Clearly, $\mD_2(w,\xi)$ satisfies \eqref{sym}-\eqref{sym1} for all $(w,\xi) \in \bar \mU \times \R^d$. In particular, $\mD_2 \in S^0(\mU)$. We show that 
$$
	K(w)=\frac12(\Op[\mD_2(w)]+\Op[\mD_2(w)]^*)+ \mF^{-1}\mD_2(0)\mF-\Op[\mD_2(0)],~w \in W^{2,\infty},
$$
is an operator with the properties asserted in Lemma \ref{lem:kv}. For $w$ depending on $t$ with $w \in C^1([0,T],L^\infty)$, \eqref{k3} directly follows from \eqref{deriv} without further assumptions.

Now, let $w \in W^{2,\infty}$ with $\|w\|_{W^{2,\infty}} < \delta$ for some $\delta>0$ to be chosen later. In the following, $C$ is a generic constant, which may depend monotonically increasing on $\delta$. First, note that by Lemma \ref{lem:fm}, $\Op[\mD_2(0)]^*=\Op[\mD_2(0)]$, i. e.,
$$
K(w)=\mF^{-1}\mD_2(0)\mF+\frac{1}{2}(\Op[\mD_2(w)-\mD_2(0)]+ \Op[\mD_2(w)-\mD_2(0)]^*).$$
In particular, $K(w)$ is self-adjoint. Lemma \ref{lem:base} and \eqref{opa} give for some $L \in \N$ and all $f \in L^2$,
$$
\|\Op[\mD_2(w)-\mD_2(0)]f\| \le C \|\mD_2(w)-\mD_2(0)\|_{0,0,L}\|f\| \le C\|w\|_{L^\infty}\|f\|.$$
Since $\mD_0$ is positive definite, we arrive at
$$
	\langle K(w)f,f \rangle \ge c\|f\|^2-C\|w\|_{L^\infty}\|f\|^2.
$$
Thus, for $\delta$ sufficiently small, $K$ satisfies \eqref{k1}. It remains to show \eqref{k2}.

By Lemmas \ref{lem:fm} and \ref{lem:base}, 
\begin{equation}
	\label{is}
	\langle (\mF^{-1}\mD_2(0)\mF-\Op[\mD_2(0)])\Op[\mL(w)]f,f \rangle \le C \|\Op[\mL(w)]f\|_{-2}\|f\|_{-1} \le C\|f\|_{-1}^2.
\end{equation}
By construction, $\mD_2(w)=\mD_2(w)^*$, and thus Proposition \ref{prop:central} (i) yields
\begin{align*}
	\|(\Op[\mD_2(w)]^*-\Op[\mD_2(w)])\Op[\mL(u)]f\|\le C\|\nabla w\|_{L^\infty}  \|f\|^2.
\end{align*}
Next, note that $\mD_2(w,\cdot)-\mD(w,\cdot)$ is compactly supported with support not depending on $x$. Together with to $\mD_2,\mD \in S^0(\mU)$, this implies  $\mD_2(w)-\mD(w) \in S^m(\mU)$ for all $m \in \R$ and hence
$$
	\|\Op[\mD_2(w)-\mD(w)]\Op[\mL(w)] f\| \le C\|f\|_{-1}^2.
$$
In conclusion
\begin{equation} 
	\label{est1}
	\begin{split}
		&\Rep\langle K(w)\Op[\mL(w)] f,f\rangle=\Rep \langle \Op[\mD(w)]\Op[\mL(w)]f,f\rangle\\
		&+\Rep\langle (\mF^{-1}\mD_2(0)\mF-\Op[\mD_2(0)])\Op[\mL(w)]f,f \rangle\\
		&+\frac12\Rep\langle(\Op[\mD_2(w)]^*-\Op[\mD_2(w)])\Op[\mL(w)]f,f\rangle+\Rep\langle \Op[\mD_2(w)-\mD(w)]\mL(w) f,f \rangle\\
		& \le \Rep \langle \Op[\mD(w)]\Op[\mL(w)]f,f\rangle+C(\|w\|_{W^{1,\infty}}\|f\|^2+\|f\|_{-1}^2)
	\end{split}
\end{equation}
By Proposition \ref{prop:central} (ii),
\begin{equation}
	\label{est2}
	\begin{split}
		\Rep \langle \Op[\mD(w)]\Op[\mL(w)]f,f\rangle &= \Rep \langle \Op[\mD(w)\mL(w)]f,f \rangle\\
		&\quad+\Rep \langle (\Op[\mD(w)]\Op[\mL(w)]-\Op[\mD(w)\mL(w)])f,f\rangle\\
		&\le\Rep \langle \Op[\mD(w)\mL(w)]f,f \rangle +C \|\nabla w\|_{L^\infty} \|f\|.
	\end{split}
\end{equation}
Next, set $\mathcal X(w):=\mD(w)\mL(w)+c_*/2I_{2n}$ with $c_*$ as in Lemma \ref{lem:d3}. Note that $c_*$  does not depend on $\delta$. 
Since $\Op[I_{2n}]-{\rm Id_{\mL^2}}$ is infinitely smoothing we conclude
\begin{equation} 
	\label{rhoxi1}
	\Rep \langle \Op[\mD(w)\mL(w)]f,f \rangle\le\Rep \langle \Op[\mathcal X_w]f,f\rangle-\frac{c_*}{2}\|f\|^2+C\|f\|_{-1}^2.
\end{equation}
By Lemma \ref{lem:d3},
$$\mathcal X(w(x),\xi)+\mathcal X(w(x),\xi)^* =\mD(w(x),\xi)\mL(w(x),\xi)+(\mD(w(x,\xi))\mL(w(x),\xi)^*+c_* \le 0,$$
for $x \in \R^d$ and $\xi \in \R^d$ with $|\xi| \ge 3r$.
Thus, Proposition \ref{prop:garding} applied to $-\mathcal X(w)$ gives 
\begin{equation}
	\label{eq:garding100}
	\Rep \langle \Op[\mathcal X(w)]f,f \rangle \le C(\|w\|_{W^{2,\infty}}^\frac12\|f\|^2+\|f\|_{-1}).
\end{equation}
Lastly, \eqref{rhoxi1} and \eqref{eq:garding100} lead to
\begin{equation}
	\label{est4}
	\Rep \langle \Op[\mD(w)\mL(w)]f,f \rangle\le-c\|f\|^2+C(\|w\|_{W^{2,\infty}}^{\frac12}\|f\|^2+\|f\|_{-1}^2)
\end{equation}
for $c$ independent of $w$. Clearly the assertion follows from \eqref{est1}, \eqref{est2}, \eqref{rhoxi1} and \eqref{est4} if $\|w\|_{W^{2,\infty}}$ is sufficiently small.

\section{Connection with relaxation systems}\label{s:connection}
We note here a connection in the semilinear case between hyperbolic-hyperbolic systems and relaxation
systems of a generalized Jin-Xin form, extending observations of \cite{JX95} 
for the original Jin-Xin model ($C^i\equiv 0$ in \eqref{semilin},\eqref{gJX} below) and
\cite[Ch. 10]{W74} in the general 1-D case.
Consider a hyperbolic regularization
\begin{equation}\label{semilin}
	\partial_t g(u) + \sum_i \partial_{x_i}f^i(u)= \sum_{ij}\partial_{x_i}(B^{ij}\partial_{x_j}u)
	-\partial_t (A\partial_t u) + \sum_i \partial_t( C^i \partial_{x_i} u)
	+ \sum_i \partial_{x_i}( \tilde C^i \partial_{t} u)
\end{equation}
of the $d$-dimensional system 
\begin{equation}\label{ddim}
	\partial_t g(u) + \sum_i \partial_{x_i}f^i(u),
\end{equation}
of {\it semilinear form}: $A$, $B^{jk}$, $C^i$, $\tilde C^i$ constant, $u\in \R^n$, with $g$ {\it linear}:
\be\label{glin}
g(u)=\alpha u,
\ee
$\alpha$ constant, without loss of generality $\alpha= I$, and data
\begin{equation}\label{data}
u|_{t=0}=\phi,\qquad \partial_t u|_{t=0}=\psi.
\end{equation}

Then, for $u$ and $v^i\in \R^n$, $i=1,\dots, d$,
satisfying the ``generalized Jin-Xin'' relaxation system
\begin{equation}\label{gJX}
	\begin{aligned}
		A \partial_t u - \sum_i (C^i+\tilde C^i) \partial_{x_i}u +  \sum_i \partial_{x_i} v^i&=0,\\
		\partial_t v^i+ \sum_{j}(B^{ij}\partial_{x_j}u) &=
		f^i(u)- A^{-1}(v^i - (C^i+\tilde C^i) u),
	\end{aligned}
\end{equation}
together with compatibility conditions
\begin{equation}\label{compatibility}
u|_{t=0}=\phi,\qquad 
	-A^{-1}\sum_i \partial_{x_i}(v^i|_{t=0}-C^i\phi)= \psi, 
\end{equation}
we find, taking without loss of generality $g(u)=u$, 
that $u$ satisfies \eqref{semilin} with data \eqref{data}, by
differentiating the first equation with respect to $t$, subtracting the sum from $i=1$ to $d$ of
the second equations differentiated with respect to $x_i$, and subtracting the first equation.

We note that \eqref{compatibility}(ii) may always be solved, but in general changes the class of
data; in particular, there may be no solution for which $v^i|_{t=0}$ decays at spatial infinity, even 
though $\psi$ does so decay, i.e., a loss of localization 
translating from hyperbolic regularization to Jin-Xin relaxation system.
In the other direction, from Jin-Xin to hyperbolic regularization,
there is loss of regularity, but not localization.
Thus, we may conclude from our nonlinear stability results for the hyperbolic regularization stability
corresponding results (with slightly stronger regularity assumed on $\psi$) for the Jin-Xin system,
but not the reverse.

Similarly, symmetry of $A$, $B^{ij}$, $C^i$ imply symmetrizability of the first-order part of \eqref{gJX}
but not simultaneous symmetrizability of the relaxation term,
hence standard structural assumptions on first-order relaxation systems are not typically 
satisfied for \eqref{gJX}.
Thus, again, despite structural similarities of the equations, {\it we cannot}, even in the semilinear case
and under condition \ref{glin},
{\it conclude stability results for hyperbolic regularization from existing literature \cite{Kw11,KwZ09} for
first-order multi-D relaxation systems.}

\begin{remark}\label{Lorentzrmk}
Condition \eqref{glin} is not Lorentz invariant, hence \emph{never} holds
in relativistic settings, in particular for relativistic gas dynamics.
More generally, in modelling physical settings, it is natural to work with symmetric coefficient matrices
in Godunov variables, for which in general $g(u)\neq u$.
Likewise, constraints of causality, existence of smooth profiles, etc., lead in the relativistic setting
to \emph{quasilinear} systems \eqref{semilin}, e.g. \cite{FT14, BDN18}, 
	for which an (analytically convenient) equivalent relaxation system seems not available.\footnote{
One can of course always rewrite \eqref{semilin} as a first-order relaxation system
in $(u,A(u)u_t,u_{x_1},\dots, u_{x_d})$, at the expense of an additional variable $A(u)u_t$,
but the resulting relaxation term does not have a constant left kernel as assumed in standard
	analyses of relaxation systems \cite{Li87,MaZ02,MaZ05,KwZ09,Kw11}.}
\end{remark}

\begin{remark}
	\label{quasilin:remark}
	One quickly checks that for a strong solutions $(u,v$) to \eqref{gJX}, \eqref{compatibility}, $u$ satisfies \eqref{semilin}, \eqref{data} even for quasilinear systems, provided $A,C^i, Dg$ are  constant. This in fact the most general case, as solutions to \eqref{semilin} need to satisfy the \textit{zero-mass condition} $\int_\R^d u(t) = 0$, for all $t$, which is invariant under the flow of \eqref{semilin} only for constant $A,C^i, Dg$. 
\end{remark}

\begin{remark}\label{sybolrmk}
	In the linear constant coefficient case, we have evidently equivalence of 
	\eqref{semilin} and \eqref{gJX}, as \eqref{glin} is always satisfied.
	This implies in particular that, \emph{even in the general quasilinear case \eqref{hyp-hyp}},
	the Fourier symbols of the linearized equations about the end states of a shock
	agree with those for a corresponding relaxation system.
	This explains in hindsight the agreement between the corresponding
	spectral structure established from first principles here and that
	established for relaxation systems in \cite{Kw11}.
\end{remark}

\subsection{As hyperbolic regularization}\label{s:regsec}
If one is only interested in a semilinear regularization of \eqref{ddim},
as in the original motivation of \cite{JX95}, then, for the 
simplest choice $\psi\equiv 0$, we have, taking $v^i|_{t=0}= C^i \phi$,
that systems \eqref{semilin} and \eqref{gJX} are completely equivalent.
Another interesting choice \cite[Eq. (1.5)]{JX95} is to initialize \eqref{gJX} at equilibrium values
$ f^i(u)- A^{-1} (v^i-C^iu)=0$, giving
$
v^i|_{t=0}= Af^i(u)-C^i u,
$
and thus 
$$
\partial_t u|_{t=0} + \sum_i \partial_{x_i}f^i(u)|_{t=0}=0,
$$
that is, initialization at equilibrium values for the second-order system \eqref{semilin} as well.
This choice, likewise, yields essential equivalence in data as well as structure of the two equations.

Indeed, any choice of $\psi=\partial_t u|_{t=0}$ that is in (spatial) divergence form
$\psi=\sum_i \partial_{x_i} \psi^i$
yields such an equivalence under the choice $ v^i|_{t=0}= C^i\phi +  A\psi^i$
satisfying \eqref{compatibility}.
Note further that we may always choose beforehand ``conservative'' variables such that $g(u)=u$,
afterward choosing the semilinear regularization; hence, \eqref{glin} for this purpose
is no real restriction.
\bigskip

\bibliographystyle{amsplain}

\end{document}